\documentclass[12pt,reqno]{amsart}
\usepackage{amssymb,amscd}
\usepackage[alphabetic]{amsrefs}
\usepackage[all]{xy}
\usepackage[headings]{fullpage}
\usepackage[T1]{fontenc}
\usepackage{libertine}
\usepackage{stackengine}
\usepackage{stmaryrd}

\numberwithin{equation}{section}

\theoremstyle{plain}
\newtheorem{thm}[subsection]{Theorem}

\newtheorem{prop}[subsection]{Proposition}

\newtheorem{lemma}[subsection]{Lemma}

\newtheorem{claim}[subsection]{Claim}

\theoremstyle{definition}
\newtheorem{defn}[subsubsection]{Definition}

\theoremstyle{remark}
\newtheorem{rem}[subsection]{Remark}
\newtheorem{rems}[subsection]{Remarks}

\numberwithin{equation}{section}

\usepackage[OT2,T1]{fontenc}
\DeclareSymbolFont{cyrletters}{OT2}{wncyr}{m}{n}
\DeclareMathSymbol{\sha}{\mathalpha}{cyrletters}{"58}

\newcommand{\CC}{\mathcal{C}}

\newcommand{\DD}{\mathcal{D}}

\newcommand{\RR}{\mathcal{R}}

\newcommand{\XX}{\mathcal{X}}

\newcommand{\WW}{{\mathcal{W}}}

\newcommand{\HH}{{\mathcal{H}}}

\renewcommand{\O}{\mathcal{O}}

\newcommand{\EE}{\mathcal{E}}

\newcommand{\LL}{\mathcal{L}}
\newcommand{\MM}{\mathcal{M}}
\newcommand{\NN}{\mathcal{N}}
\newcommand{\OO}{\mathcal{O}}
\newcommand{\PP}{\mathcal{P}}

\newcommand{\Fp}{{\mathbb{F}_p}}

\newcommand{\Z}{\mathbb{Z}}

\newcommand{\Q}{\mathbb{Q}}
\newcommand{\R}{\mathbb{R}}
\newcommand{\C}{\mathbb{C}}
\newcommand{\A}{\mathbb{A}}
\renewcommand{\P}{\mathbb{P}}

\newcommand{\into}{\hookrightarrow}

\newcommand{\tensor}{\otimes}

\newcommand{\nodiv}{\not|}
\def\nodiv{\mathrel{\mathchoice{\not|}{\not|}{\kern-.2em\not\kern.2em|}
{\kern-.2em\not\kern.2em|}}}

\newcommand{\G}{\mathbb{G}}

\DeclareMathOperator{\ord}{ord}

\DeclareMathOperator{\dvsr}{div}

\DeclareMathOperator{\aut}{Aut}
\DeclareMathOperator{\Pic}{Pic}

\DeclareMathOperator{\NS}{NS}

\DeclareMathOperator{\spec}{Spec}
\DeclareMathOperator{\proj}{Proj}

\def\clap#1{\hbox to 0pt{\hss#1\hss}}

\begin{document}
\title[Transversality of sections on elliptic surfaces]{Transversality
  of sections on elliptic surfaces\\ 
with applications to elliptic divisibility sequences\\
and geography of surfaces}

\author{Douglas Ulmer}
\address{Department of Mathematics \\ University of Arizona
  \\ Tucson, AZ~~85721 USA}
\email{ulmer@math.arizona.edu}

\author{Giancarlo Urz\'ua}
\address{Facultad de Matem\'aticas \\ Pontificia Universidad
  Cat\'olica de Chile \\ Santiago, Chile}
\email{urzua@mat.uc.cl}

\date{\today}


\begin{abstract}
  We consider elliptic surfaces $\mathcal{E}$ over a field $k$
  equipped with zero section $O$ and another section $P$ of infinite
  order.  If $k$ has characteristic zero, we show there are only
  finitely many points where $O$ is tangent to a multiple of $P$.
  Equivalently, there is a finite list of integers such that if $n$ is
  not divisible by any of them, then $nP$ is not tangent to $O$.  Such
  tangencies can be interpreted as unlikely intersections.  If $k$ has
  characteristic zero or $p>3$ and $\mathcal{E}$ is very general, then
  we show there are no tangencies between $O$ and $nP$.  We apply
  these results to square-freeness of elliptic divisibility sequences
  and to geography of surfaces.  In particular, we construct mildly
  singular surfaces of arbitrary fixed geometric genus with $K$ ample
  and $K^2$ unbounded.
\end{abstract}

\maketitle

\section{Introduction}
Our aim in this paper is to study transversality properties of
sections of elliptic surfaces and to deduce consequences for elliptic
divisibility sequences and geography of surfaces.

To state the first result, let $k$ be a field of characteristic zero
and let $\CC$ be a smooth, projective, geometrically irreducible curve
over $k$.  Let $\pi:\EE\to\CC$ be a relatively minimal Jacobian
elliptic surface over $k$ (i.e., a smooth elliptic surface with a
section $O$ which will play the role of zero section), and let $P$ be
another section.  We write $nP$ for the section induced by
multiplication by $n$ in the group law of the fibers of $\EE\to\CC$.
Assume that $P$ has infinite order, i.e., $nP\neq O$ for all $n\neq0$.
As we will see below, except in degenerate situations the intersection
number $(nP).O$ grows like a constant times $n^2$.  Our first result
says that the intersections are usually transverse.

\begin{thm}\label{thm:one-surface'}
The set
\[T=\bigcup_{n\neq0}\left\{t\in\CC\left|\text{ $nP$ is tangent to $O$
      over $t$}\right.\right\}\]
is finite.
\end{thm}

Here and in the rest of the paper, we conflate the sections
$O:\CC\to\EE$ and $P:\CC\to\EE$ with their images $O(\CC)\subset\EE$
and $P(\CC)\subset\EE$.  Thus we say ``$P$ is tangent to $O$'' rather
than ``the image of $P$ is tangent to the image of $O$.''

\begin{rem}
  We note that a tangency between $nP$ and $O$ can be regarded as an
  ``unlikely intersection'' as follows: Let $T_\EE$ be the tangent
  bundle of $\EE$ and let $\P T_\EE$ be the associated projective
  bundle.  Thus $\P T_\EE\to\EE$ is a $\P^1$-bundle, and the total
  space $\P T_\EE$ is a smooth, projective threefold.  If
  $C\subset\EE$ is a smooth curve, then there is a canonical lift of
  $C$ to $\tilde C\subset\P T_\EE$ defined by sending a point $t\in C$
  to the class of its tangent line $T_{C,t}\subset T_{\EE,t}$ in
  $\P T_\EE$.  Two curves $C_1$ and $C_2$ in $\EE$ that meet at
  $y\in\EE$ are tangent there if and only if their lifts meet at a
  point of $\P T_\EE$ over $y$.  Thus a tangency between $C_1$ and
  $C_2$ is equivalent to the ``unlikely'' intersection of the two
  curves $\tilde C_1$ and $\tilde C_2$ in the threefold $\P T_\EE$.
  We refer to \cite{ZannierUnlikely12} for a comprehensive account of
  work on unlikely intersections up to 2012.
\end{rem}

\begin{rems}\mbox{}
  \begin{enumerate}
  \item A result very similar to our Theorem~\ref{thm:one-surface'}
    was communicated to us by Corvaja, Demeio, Masser, and Zannier
    after we posted the first version of this paper.  Their methods
    are rather different, see \cite{CorvajaDemeioMasserZannierpp}.
    They show more generally that finiteness holds when the cyclic
    group $\{nP|n\in\Z\}$ is replaced by a finitely generated,
    torsion-free group of sections.
  \item On the other hand, our methods lead to non-trivial results in
    families, and in particular we show that for ``generic'' data, the
    set $T$ above is empty.  (See Theorems~\ref{thm:vg} and
    \ref{thm:explicit} below.)  This is crucial for our application to
    geography of surfaces.
  \item In the first version of this paper, we used a
    trivialization essentially equivalent to the Betti foliation
    discussed in Section~3 of this version. Later, we learned of the
    ``Betti'' terminology used by several authors, including in
    \cite{CorvajaMasserZannier18}, and adopted it in this paper.
  \end{enumerate}
\end{rems}

We next reformulate Theorem~\ref{thm:one-surface'} in analogy
with the ``elliptic divisibility sequence'' associated to an elliptic
curve and a point.  (See \cite[Exers. 3.34-36, 9.4,
9.12]{SilvermanAEC} for definitions and examples, and
\cite{Ingram-etal12} for more on the function field case.)  Define a
sequence of effective divisors on $\CC$ for $n\ge1$ by
\[ D_n:=O^*(nP),\] 
i.e., $D_n$ is the pull-back along the zero section of the divisor
$nP$ on $\EE$.  (We will give several other equivalent definitions in
Section~\ref{s:prelims}.) 

The sequence $D_n$ is a natural analogue of an elliptic divisibility
sequence.  In particular, we will see below that if $m$ divides $n$,
then $D_m$ divides $D_n$ (i.e., $D_n-D_m$ is effective), and that
M\"obius inversion gives a sequence of effective divisors $D'_m$ such
that
\[D_n=\sum_{m|n}D'_m.\]

We say that a divisor on $\CC$ is \emph{reduced} if it has the form
\[ D=\sum_it_i\] 
where the $t_i$ are distinct closed points of $\CC$ (i.e., each
non-zero coefficient of $D$ equals $1$).  This is an analogue of an
integer being square-free.

\begin{thm}\label{thm:one-surface}
Given $\EE$ and $P$ as above, there is a finite set of integers
$M=\{m_1,\dots,m_k\}$ such that
\begin{enumerate}
\item $O$ and $nP$ intersect transversally if and only if $n$ is not
  divisible by any  element of $M$. 
\item $D_n$ is reduced if and only if $n$ is not divisible by any
  element of $M$.
\item $D'_m$ is reduced if and only if $m\not\in M$.
\end{enumerate}
\end{thm}

\begin{rem}
  Thereom~\ref{thm:one-surface} is much stronger than what one might
  predict from standard conjectures.  For simplicity, assume that
  $\CC=\P^1$, let $F=\C(\CC)$, and let $E/F$ be the generic fiber of
  $\EE\to\CC$.  Choosing a coordinate $t$ on $\P^1$ so that none of
  the $D_n$ involve the place at infinity, we may identify each $D_n$
  with a monic polynomial $f_n$ in $t$, and to say that $D_n$ is
  reduced is to say that $\gcd(f_n,df_n/dt)=1$.  Arguments similar to
  those in \cite{Silverman05} applied to a certain Buium jet space of $E/F$
  together with a function field analogue of Vojta's conjecture
  suggest that 
  \[\deg\gcd(f_n,df_n/dt)\overset{?}{=}o(n^2)=o(h(nP))\]
  where $h(nP)$ is the canonical height of $nP$.  (See
  Section~\ref{ss:heights} for definitions.)  But
  Theorem~\ref{thm:one-surface} shows that
  $\deg\gcd(f_n,df_n/dt)$ is in fact bounded!
\end{rem}

\begin{rem}
  We have no reason to believe that the analogues of
  Theorems~\ref{thm:one-surface'} and \ref{thm:one-surface} (with $n$
  restricted to be prime to the characteristic) are false in positive
  characteristic.  However, our proof uses analytic techniques and
  does not obviously carry over to the arithmetic situation.
\end{rem}

The next two results hold for $k$ a field of characteristic zero or
sufficiently large $p$.  As before, $\CC$ is a smooth, projective,
geometrically irreducible curve over $k$.  The next result says
roughly that if $\EE\to\CC$ is a very general Jacobian elliptic
surface with an additional section $P$, there are no tangencies
between $nP$ and $O$ for $n\neq0$.  Recall that a line bundle $L$ on
$\CC$ is said to be globally generated (or base point free) if for
every $t\in\CC$, there is a global section of $L$ which does not
vanish at $t$.

\begin{thm}\label{thm:vg}
  Suppose $k$ is a field of characteristic zero or of characteristic
  $p>3$, and let $\CC$ be as above.  Let $L$ be a globally generated
  line bundle on $\CC$ of degree $d$ and set
\[V=H^0(L^2\oplus L^3\oplus L^4).\]  
Then for a very general $a=(a_2,a_3,a_4)\in V$, the elliptic surface
$\EE\to\CC$ associated to
\[E:\qquad y^2+a_3y=x^3+a_2x^2+a_4x\]
equipped with the usual zero section $O$ and the section $P=(0,0)$ has
the following properties: 
\begin{enumerate}
\item $P$ has infinite order.
\item The singular fibers of $\EE\to\CC$ are nodal cubics
  \textup{(}i.e., Kodaira type $I_1$\textup{)}.
\item $P$ meets each singular fiber in a non-torsion point.
\item If $n$ is not a multiple of the characteristic of $k$, then $nP$
  meets $O$ transversally in $d(n^2-1)$ points. 
\end{enumerate}
\end{thm}

(Here and elsewhere in the paper, $L^i$ means $L^{\tensor i}$ and not
$L^{\oplus i}$.)  We will explain the construction of the elliptic
surface attached to $a$ in Section~\ref{ss:L-to-EE} and the
meaning of ``very general'' in Section~\ref{s:v-g-surfaces}.

As with many results about ``very general'' points,
Theorem~\ref{thm:vg} does not allow one to deduce the existence of
examples over ``small'' (countable) fields such as number fields or global
function fields.   However, after relaxing condition (2) above, we can
write down such examples explicitly, at least when $L$ is the square
of a globally generated line bundle.

\begin{thm}\label{thm:explicit}
  Let $k$ be a field of characteristic 0 or a field of characteristic
  $p>2$ which is not algebraic over the prime field $\Fp$.  Let $\CC$
  be a smooth, projective, geometrically irreducible curve over $k$
  with a non-trivial line bundle $L$ which is the square of a globally
  generated line bundle $F$.  Then there exist infinitely many pairs
  $(\EE,P)$ where $\EE$ is a Jacobian elliptic surface $\EE\to\CC$
  equipped with a section $P$ such that:
\begin{enumerate}
\item $P$ has infinite order.
\item The singular fibers of $\EE\to\CC$ are of Kodaira type $I_0^*$.
\item $P$ meets each singular fiber in a non-torsion point.
\item If $n$ is not a multiple of the characteristic of $k$, then $nP$
  meets $O$ transversally in 
 \[\begin{cases}
\displaystyle\frac{n^2-1}2d&\text{if $n$ is odd,}\\
\displaystyle\frac{n^2-4}2d&\text{if $n$ is even}
\end{cases}\]
points, where $\deg L=2d$.
\item $O^*(\Omega^1_{\EE/\CC})\cong L$.
\end{enumerate}
\end{thm}

The starting point for our collaboration was a remarkable application
of Theorem~\ref{thm:vg} to the geography of surfaces due to the
second-named author.  We give some background before stating the
result: There has been a great deal of interest in volumes of stable
surfaces, i.e., in the set of values of $K_X^2$ where $X$ runs through
stable surfaces in the sense of Koll\'ar, Shepherd-Barron, and Alexeev
(KSBA).  A key ingredient in the construction of the KSBA
compactifications of moduli spaces of surfaces of general type (see
\cite{KollarShepherd-Barron88}) is a descending chain condition on the
set of $K_X^2$, namely that it admits no strictly decreasing
sequences.  The famous article \cite{Alexeev94} establishes this
condition.  (See also \cite{AlexeevMori04}.)  As this is a set of
positive rational numbers, it must have a minimum, whose value is
still unknown.  Aleexev and Liu have recently found various other
special properties around accumulation points and bounds
\cites{AlexeevLiu19c, AlexeevLiu19b, AlexeevLiu19a, Liu-pp17}.  For
example, \cite{Liu-pp17} shows that when the geometric genus is
positive then the volume can be optimally bounded from below.
%(And so the open question about explicit minimal
%volume is about geometric genus zero surfaces, in fact, rational
%surfaces.)
This naturally raises the question: Are there \emph{upper}
bounds for $K^2$ of stable surfaces with a fixed geometric genus?

Recall \cite[I.5.5, VI.1]{BarthHulekPetersVandeVenCCS} that a smooth,
projective, non-ruled surface $X$ over a field of characteristic zero
has $c_2(X)\ge0$, so Noether's formula shows that $K_X^2\le12(1+p_g)$,
i.e., the self-intersection of the canonical bundle $K_X$ is bounded
in terms of the geometric genus $p_g$.  The question is whether such
bounds continue to hold for mildly singular surfaces.

Any geometric approach to a negative answer to this question via
surfaces with rational singularities requires a family of
minimal resolution surfaces for any given $p_g>0$ with an unbounded
number of special rational curves. Our work supplies such families and
moreover allows for good control on the singularities involved:

\begin{thm}\label{thm:geography}
  Given integers $g\ge0$ and $N$, there exists a normal
  projective surface $X$ over $\C$ with the following properties:
  \begin{enumerate}
  \item $X$ has geometric genus $p_g=g$.
  \item $X$ has only one singular point, which is log-terminal.
  \item $K_X$ is $\Q$-Cartier and ample.
  \item $K_X^2 > N$.
  \end{enumerate}
\end{thm}

\subsection{Plan of the paper}  In Section~\ref{s:prelims} we present
foundational material on torsion points and intersections on elliptic
surfaces, including a discussion of basic properties of our elliptic
divisibility sequences.  We then reformulate
Theorems~\ref{thm:one-surface'} and \ref{thm:one-surface} as
Theorem~\ref{thm:discreteness}.  We prove
Theorem~\ref{thm:discreteness} in Section~\ref{s:one-surface}.
Section~\ref{s:moduli} discusses two moduli spaces which play a key
role in the proof of Theorem~\ref{thm:vg}.
Section~\ref{s:constructing-EE} discusses a construction of elliptic
surfaces equipped with extra structure associated to elements in
certain Riemann-Roch spaces.  We then prove Theorem~\ref{thm:vg} in
Section~\ref{s:v-g-surfaces}.  In Section~\ref{s:explicit}, we give an
explicit construction of surfaces satisfying the requirements of
Theorem~\ref{thm:explicit}.  Finally, in Section~\ref{s:geography}, we
prove Theorem~\ref{thm:geography}.

\subsection{Acknowledgements}
The first-named author thanks Seoyoung Kim, Nicole Looper, and Joe
Silverman for helpful conversations at the 2019 AMS Mathematics
Research Community meeting in Whispering Pines, Rhode Island, and for
pointing out \cite{GhiocaHsiaTucker18} and its antecedents.  He also
thanks the Simons Foundation for partial support in the form of
Collaboration Grant 359573.  The second-named author thanks FONDECYT
for support from grant 1190066.  Both authors thank Matthias Sch\"utt
for comments and corrections, and they thank Pietro Corvaja, Brian
Lawrence, and Umberto Zannier for their comments on an earlier version
of this paper and their pointers to related literature, notably the
preprint \cite{CorvajaDemeioMasserZannierpp}.

\section{Preliminaries on torsion and intersections}\label{s:prelims}
In this section we gather various foundational results on torsion,
intersections, heights, and elliptic divisibility sequences.  Some of
this material also appears in \cite{Ingram-etal12}, although our point
of view is more geometric.  Throughout, $\pi:\EE\to\CC$ will be a
relatively minimal Jacobian elliptic surface over a field $k$ with
zero section $O$.

\subsection{Multiplication by $n$}\label{ss:mult-by-n}
Let $\EE^{sm}$ denote the locus where $\pi$ is smooth (i.e., the
complement of the singular points in the bad fibers).  Then by
\cite[Prop.~II.2.7]{DeligneRapoport73}, $\EE^{sm}$ is a commutative group
scheme over $\CC$.  Let $n$ be an integer not divisible by the
characteristic of $k$.  Consider the homomorphism of group schemes
given multiplication by $n$: $[n]:\EE^{sm}\to\EE^{sm}$.

Clearly, $[n]$ fixes the zero section $O$ pointwise.  If $x\in O$, the
tangent space to $\EE^{sm}$ at $x$ splits canonically into the sum of two
lines, namely the tangent space to $O$ at $x$ and the tangent space to
the fiber of $\pi$ through $x$.  Since $[n]$ fixes $O$, $[n]$ acts as
the identity on the former.  A calculation in the formal group of
$\EE^{sm}$ \cite[Ch.~IV]{SilvermanAEC} shows that $[n]$ acts as
multiplication by $n$ on the tangent space to the fiber of $\pi$ through
$x$.  It follows that $[n]$ is \'etale at every point of $O$, and since
$[n]$ is a group scheme homomorphism it is \'etale everywhere.

The morphism $[n]$ is also quasi-finite: it has degree $n^2$ on the
smooth geometric fibers of $\pi$, and degree dividing $n^2$ on all
geometric fibers \cite[Ch.~III and \S VII.6]{SilvermanAEC}.  It is not
in general finite if $\pi$ has singular fibers.

If $P:\CC\to\EE$ is a section of $\pi$, $P$ necessarily lands in the
smooth locus $\EE^{sm}$ and we may define a new section $nP$ as the
composition $[n]\circ P$.  This is the meaning of the notation $nP$
used in the introduction.  

\subsection{Torsion}
With $\EE$ as  above and $n>0$ and relatively prime to the
characteristic of $k$, we define
\[\EE[n]=[n]^{-1}(O),\]
i.e, $\EE[n]$ is the inverse image the zero section under $[n]$.
Since $[n]$ is \'etale, $\EE[n]$ is a reduced, closed subscheme of
$\EE^{sm}$ of dimension 1, and in particular, locally closed in $\EE$.
Since $[n]$ is quasi-finite, $\EE[n]$ is \'etale and quasi-finite over
$\CC$ of generic degree $n^2$.  It is in general not finite over
$\CC$.

The fiber of $\EE[n]$ over a geometric point $t$ of $\CC$ consists of the
points of $\pi^{-1}(t)$ with order divisible by $n$.  We define
\[\EE[n]'\subset\EE[n]\]
to be the subscheme whose fiber over $t$ consists of the  points of
$\pi^{-1}(t)$ of order exactly $n$.  If $m$ divides $n$, then
$\EE[m]'$ is a closed subscheme of $\EE[n]$, and we have a
disjoint union
\begin{equation}\label{eq:EEn-union}
\EE[n]=\cup_{m|n}\EE[m]'  
\end{equation}
where $m$ runs over the positive divisors of $n$.  Each $\EE[m]'$ is a
union of irreducible components of $\EE[n]$ and is \'etale and
quasi-finite over $\CC$.   Note that $\EE[1]=\EE[1]'=O$.   

We refer to the unions of irreducible components of $\EE[n]$ as
``torsion multisections''.

\subsection{Divisibility sequences}
In the introduction, we defined divisors $D_n$ for $n\ge1$ by
\[D_n=O^*(nP).\]  
In this section we examine alternative definitions and properties of
these divisors always assuming that $n$ is relatively prime to the
characteristic of $k$.

For two smooth curves $C_1$ and $C_2$ on $\EE$ with no irreducible
components in common, write $C_1\cap C_2$ for the intersection
zero-cycle.  This is a zero-dimensional closed subscheme of $\EE$.
With this notation,
\[D_n=\pi_*\left(nP\cap O\right)=\pi_*\left(nP\cap \EE[1]\right).\]

Note that $nP$ meets $O=\EE[1]$ over $t$ if and only if $P$ meets
$\EE[n]$ over $t$, and since $[n]$ is \'etale, the intersection
multiplicity of $nP$ and $\EE[1]$ over $t$ equals the intersection
multiplicity  of $P$ and $\EE[n]$ over $t$.  In other words, we have
\begin{equation}\label{eq:Dn-alt}
  D_n=\pi_*\left(P\cap \EE[n]\right).
\end{equation}

Define
\[D_n'=\pi_*\left(P\cap \EE[n]'\right).\]
Then the disjoint union \eqref{eq:EEn-union} yields a decomposition of
$D_n$ into effective divisors:
\begin{equation}\label{eq:Dn-sum}
D_n=\sum_{m|n}D_m'  
\end{equation}
where the sum is over positive divisors of $n$.

Note in particular that if $t$ is a closed point of $\CC$ and $P(t)$
is a torsion point, say of order exactly $m$, then $t$ appears in $D_n$
if and only if $m$ divides $n$, and the multiplicity of $t$  in such
$D_n$ is equal to the multiplicity of $t$ in $D_m'$.  

\begin{rem}
  A section $P$ can meet at most one torsion point over a given
  $t\in\CC$.  This implies that if $m_1\neq m_2$, then $D_{m_1}'$ and
  $D_{m_2}'$ have disjoint support.  In particular, as soon as
  $D_n'\neq0$, $D_n$ has ``primitive divisors' i.e., points in its
  support which are not in the support of $D_m$ for $m<n$.  The
  existence of primitive divisors for all sufficiently large $n$ is
  established in \cite[\S5]{Ingram-etal12} by showing that $D_n'\neq0$
  for all sufficiently large $n$.  The key idea is an estimation of
  intersection numbers using heights as in Section~\ref{ss:heights}
  below.  Another simple proof of the existence of primitive divisors
  (suggested by a referee) can be given using the fact that the
  ``Betti coordinates'' (as in Section~\ref{s:one-surface}) of a
  non-torsion section $P$ give rise to a locally defined, open map
  from $\CC$ to $\R^2$.  The existence and openness of this map can
  also be viewed as the key point in the simplest case ($I_0$) of the
  proof of Theorem~\ref{thm:discreteness}.
\end{rem}

We now state a result which implies Theorems~\ref{thm:one-surface'}
and \ref{thm:one-surface}:

\begin{thm}\label{thm:discreteness}
  Let $\EE\to\CC$ be a relatively minimal Jacobian elliptic
  surface over the complex numbers $\C$, with zero section $O$ and
  another section $P$ which is not torsion.  Then the set
\[T_{tor}:=\bigcup_{n\neq0}\left\{t\in\CC\left|\text{ $P$ is tangent to $\EE[n]$
      over $t$}\right.\right\}\]
is finite.
\end{thm}

\subsection{Proof that Theorem~\ref{thm:discreteness} implies
Theorems~\ref{thm:one-surface'} and \ref{thm:one-surface}}
First we note that the general cases of Theorems~\ref{thm:one-surface'}
and \ref{thm:one-surface} follow from the case $k=\C$.  Indeed, since
the hypotheses and conclusion of Theorems~\ref{thm:one-surface'} and
\ref{thm:one-surface} are insensitive to the ground field, we may
replace $k$ with a subfield $k'$ which is finitely generated over $\Q$
(take the field generated by the coefficients defining $\CC$, $\EE$,
$\pi$, and $P$), then embed $k'$ in $\C$.  Thus it suffices to treat
the case $k=\C$.

Next, by the definition of $D_n$, to say that $nP$ is tangent to $O$
over $t$ is to say that $t$ appears in $D_n$ with multiplicity greater
than 1.  By the equality \eqref{eq:Dn-alt}, to say that $t$ appears in
$D_n$ with multiplicity greater than 1 is to say that $P$ is tangent
to $\EE[n]$ over $t$.  Thus the set $T_{tor}$ of
Theorem~\ref{thm:discreteness} is equal to the set $T$ of
Theorem~\ref{thm:one-surface'}, and Theorem~\ref{thm:discreteness} is
equivalent to the case $k=\C$ of Theorem~\ref{thm:one-surface'}.

To finish, we show that Theorems~\ref{thm:one-surface'} and
\ref{thm:one-surface} are equivalent.  First note that points (1) and
(2) of Theorem~\ref{thm:one-surface} are equivalent by the definition
of $D_n$.  Moreover, $D_m'$ is non-reduced if and only if $D_n$
is non-reduced for all multiples $n$ of $m$.  Thus point (3) of
Theorem~\ref{thm:one-surface} implies points (1) and (2), and 
Theorem~\ref{thm:one-surface} is equivalent to the statement that the
set of $m$ such that $D_m'$ is non-reduced is finite.

Now consider the ``incidence correspondence'' 
\[I:=\left\{(t,m)\left|\ m>0\text{ and }P\text{ is tangent to
      }\EE[m]'\text{ over }t\right.\right\} \subset\CC\times\Z_{>0}.\]
The set $T$ of Theorem~\ref{thm:one-surface'} is the image of the
projection $I\to\CC$ and the set $M$ of Theorem~\ref{thm:one-surface}
is the image of the projection $I\to\Z_{>0}$.  The fibers of $I\to\CC$
are finite (and in fact empty or singletons) because $P$ meets
$\EE[m]'$ for at most one value of $m$ and \emph{a fortiori} can be
tangent to at most one $\EE[m]'$.  The fibers of $I\to\Z_{>0}$ are
finite because for a fixed $m$, $P$ meets $\EE[m]'$ at only finitely
many points, so \emph{a fortiori} can be tangent to $\EE[m]'$ at only
finitely many points.  This establishes that
Theorem~\ref{thm:one-surface'} and Theorem~\ref{thm:one-surface} are
equivalent, and it completes the proof that
Theorem~\ref{thm:discreteness} implies Theorems~\ref{thm:one-surface'}
and \ref{thm:one-surface} \qed

We will prove Theorem~\ref{thm:discreteness} in
Section~\ref{s:one-surface}.  First, we review material on heights
used later in the paper.

\subsection{Heights}\label{ss:heights}
We refer to \cite{CoxZucker79} or \cite{Shioda90} or \cite{Shioda99}
or \cite{Ulmer13a} for the basic assertions on heights in this
section.  As usual, $\pi:\EE\to\CC$ is a relatively minimal Jacobian
elliptic surface over a field $k$.

Given a section $P$ of $\pi$, there is a unique $\Q$-divisor $C_P$
supported on the non-identity components of the fibers of $\pi$ with
the property that $P-O+C_P$ has zero intersection multiplicity with
every irreducible component of every fiber of $\pi$.  There is a
simple recipe for $C_P$ that depends only on the components of the
reducible fibers met by $P$, and in particular, for a fixed $\EE$,
there are only finitely many possibilities for $C_P$ as $P$ varies
over all sections.  If $\pi$ has irreducible fibers, or more
generally, if $P$ passes through the identity component of every
fiber, then $C_p=0$.

There is a canonical $\Q$-valued symmetric bilinear form on the group of
sections of $\EE$ defined by
\begin{equation}\label{eq:ht-def}
  \langle P,Q\rangle:=-(P-O+C_P).(Q-O)
\end{equation}
where the dot refers to the intersection number on $\EE$. If
$\EE\to\CC$ is non-constant (i.e., is not isomorphic over
$\overline{k}$ to a product $E_0\times\CC$), then this pairing is
non-degenerate modulo torsion and positive definite.  We define
$ht(P)=\langle P,P\rangle$.  (Note that this is twice the height
considered in \cite{Ingram-etal12}.)

\begin{lemma}\label{lemma:int-numbers}
For $\EE$ and $P$ as above,
\[(nP).O=\frac{ht(P)}2n^2+O(1).\]
If $\pi$ has irreducible fibers and $d=\deg O^*(\Omega^1_{\EE/\CC})$,
then
\[(nP).O=\frac{ht(P)}2n^2-d.\]
\end{lemma}

\begin{proof}
It follows from the canonical bundle formula for elliptic surfaces,
adjunction, and the definition of $d$ that every section $P$ of
$\pi$ satisfies $P^2=-d$.  From the height formula, we find
\[n^2ht(P)=ht(nP)=-(nP-O+C_{nP}).(nP-O),\]
and so 
\[(nP).O=\frac{ht(P)}2n^2-d+C_{nP}.(nP-O)
=\frac{ht(P)}2n^2+O(1).\]
If $\pi$ has irreducible fibers, then $C_Q=0$ for all $Q$ and we
deduce the stated exact formula.
\end{proof}

In the following lemma we use square brackets to indicate the class of
a curve in $\NS(\EE)$, the N\'eron-Severi group of $\EE$.  This allows
us to distinguish between $n[P]$ ($n$ times the class of $P$) and
$[nP]$ (the class of $nP$).

\begin{lemma}\label{lemma:nP-in-pic}
Suppose that $\pi:\EE\to\CC$ has irreducible fibers, $d=\deg
O^*(\Omega^1_{\EE/\CC})$, and $P$ is a section of $\pi$ which does not
meet $O$.  Let $F$ be a fiber of $\pi$.  Then we have an equality
\[[nP]=n[P]+(1-n)[O]+d(n^2-n)[F]\]
in $\NS(\EE)$.
\end{lemma}

\begin{proof}
We have an equality $[nP]-[O]=n([P]-[O])$ in the Picard group of the
generic fiber of $\EE$, so there is an equality of the form 
\[[nP]=n[P]+(1-n)[O]+c[F]\]
in $\NS(\EE)$, and we just need to determine the coefficient of $[F]$.
We do this by intersecting with $[O]$.  By assumption $[P].[O]=0$, so
$ht(P)=2d$.  By the previous lemma, $[nP].[O]=d(n^2-1)$ and solving
for $c$ yields $c=d(n^2-n)$.
\end{proof}

\section{Proof of Theorem~\ref{thm:discreteness}}\label{s:one-surface} 
We first note that Theorem~\ref{thm:discreteness} is a statement about
intersections on an elliptic surface over the complex numbers.  To
prove it, we may replace $\EE$ and $\CC$ with the corresponding compex
manifolds and make use of the classical topology, i.e., the
topology induced by the metric topology on $\C$.  For the rest of this
section, we make this replacement, although we will not change the
notation.  

Our strategy will be to show that the subset $T_{tor}\subset\CC$ is
discrete (in the classical topology). This implies
Theorem~\ref{thm:discreteness} since $\CC$ is compact.  We note in
passing that the set of points of intersection of $P$ and $\EE[n]$ for
varying $n$ is usually everywhere classically dense in $P$, so the
discreteness that lies at the heart of the theorem is not evident.

In fact, will consider a more general set of tangencies and prove that
a certain subset $T_{Betti}\subset\CC$ is discrete (and thus finite)
and contains all points of $T_{tor}$ over which $\EE$ has good
reduction.  Since the set of points of bad reduction is finite, this
will establish that $T_{tor}$ is also discrete and finite.

We will establish the desired discreteness by using the complex
analytic description of $\EE\to\CC$ given by Kodaira in
\cite[\S8]{Kodaira63}.  Let $\CC^0\subset\CC$ be the maximal open
subset over which $\EE$ has good reduction, and let
$\EE^0=\pi^{-1}(\CC^0)$.

Let $t\in\CC^0$.  Write $\HH$ for the upper half plane.  Then there is
a neighborhood $U$ of $t$ biholomorphic to a disk $\Delta$ and
holomorphic functions $\tau:\Delta\to\HH$ and $w:\Delta\to\C$ such
that $\pi^{-1}(U)\to U$ sits in a diagram
\begin{equation}\label{eq:triv1}
\xymatrix{\pi^{-1}(U)\ar[d]\ar[r]
&(\Delta\times\C)/(\Z\tau+\Z)\ar[d]\\
U\ar@/^/[u]^{P_{|U}}\ar[r]
&\Delta\ar@/_/[u]_{[w]}}
\end{equation}
where the horizontal maps are biholomorphic, and
$(\Delta\times\C)/(\Z\tau+\Z)$ means the quotient of
$\Delta\times\C$ by $\Z^2$ acting as
\[(a,b)(z,w)=\left(z,w+a\tau(z)+b\right).\] 
For $z\in\Delta$, corresponding to $u\in U$, $P(u)$ corresponds to
$[w](z)$, which is the class of $w(z)$ in
$\{z\}\times\C/(\Z\tau(z)+\Z)$.  We also assume $t\in U$ corresponds
to $0\in\Delta$.

Next, we consider a trivialization of $\pi^{-1}(U)\to U$ as a real
analytic manifold.  Introduce real coordinates as follows: $z=x+iy$ on
the base $\Delta$, $\tau=\rho+i\sigma$ on the upper half plane $\HH$,
and $w=u+iv$ in the $\C$ which uniformizes the fibers of
$\pi^{-1}(U)\to U$.  Let $(r,s)$ be coordinates on $\R^2$, and
note that $w=r\tau+s$ if and only if $r=v/\sigma$ and
$s=(u\sigma-v\rho)/\sigma$.

Consider the diagram
\begin{equation}\label{eq:triv2}
\xymatrix{\Delta\times\C\ar[rr]\ar[d]
&&\Delta\times\R^2\ar[d]\\
(\Delta\times\C)/(\Z\tau+\Z)\ar[rr]\ar[rd]
&&\Delta\times(\R/\Z)^2\ar[ld]\\
&\Delta&}
\end{equation}
where the upper horizontal map is
\[(z,w)=(x+iy,u+iv)\mapsto
  (z,r,s)=\left(z,\frac{v}{\sigma},\frac{u\sigma-v\rho}{\sigma}\right),\]
with inverse
\[(z,w)=(z,r\tau(z)+s)\mapsfrom(z,r,s).\]
The two vertical maps are the natural quotients, the middle horizontal
map is induced by the upper horizontal map, and the diagonal maps are
the projections to the first factor.

The top horizontal map is a real-analytic isomorphism which is
$\R$-linear on each fiber of the projection to $\Delta$.  The choice of
this map is motivated by the fact that torsion sections of $\EE$ over
$U$ correspond the surfaces
$\Delta\times(r,s)\subset\Delta\times(\R/\Z)^2$ where $r$ and $s$ are
rational numbers.  In other words, we have changed coordinates so that
every torsion section becomes a constant section.  It will
be of interest to consider all the horizontal sections
$\Delta\times(r,s)$ for arbitrary real numbers $r$ and $s$.

\begin{defn}
  With notation as above and $(r,s)\in\R^2$, define the \emph{local
    Betti leaf} $\LL_{r,s}\subset\pi^{-1}(U)$ as the image of the map
  $U\to(\Delta\times\C)/(\Z\tau+\Z)\cong\pi^{-1}(U)$ sending $z$ to
  the class of $(z,r\tau(z)+s)$.
\end{defn}

This terminology is inspired by \cite{CorvajaMasserZannier18}, where
$r$ and $s$ are called ``Betti coordinates''.  Clearly the assignment
$(r,s)\mapsto\LL_{r,s}$ factors through $(\R/\Z)^2$.  Each $\LL_{r,s}$
is a closed holomorphic submanifold of $\pi^{-1}(U)$, and the set of
$\LL_{r,s}$ as $(r,s)$ runs through $(\R/\Z)^2$ is a foliation of
$\pi^{-1}(U)$.  Although the indexing of $\LL_{r,s}$ by $(r,s)$
depends on the choice of period map $\tau$, the submanifolds
$\LL_{r,s}$ themselves are intrinsic (i.e., independent of $\tau$).
There is a corresponding global foliation of $\EE^0$ which we will not
consider in this paper, except implicitly in the following remark: If
for some non-empty open $U\subset\CC^0$ and some $(r,s)$,
$P(U)=\LL_{r,s}$ (i.e., $P$ lands in a leaf of the local foliation),
then by analytic continuation, the same holds over every open.  In
this case, we say ``$P$ lies in the Betti foliation''.

Note that if $(r,s)\in\Q^2$, then each point of $\LL_{r,s}$ is a
torsion point in its fiber.  More precisely, if $n$ is the smallest
positive integer such that $(nr,ns)\in\Z^2$, then $\LL_{r,s}$ is a
connected component of $\EE[n]'\cap\pi^{-1}(U)$.  On the other hand,
if $(r,s)\in\R^2\setminus\Q^2$, then $\LL_{r,s}$ is disjoint from
every $\EE[n]$.  Thus, if $P$ lies in the Betti foliation, it is a
torsion section (which we have ruled out by hypothesis) or it meets no
torsion sections over $\CC^0$ and so is not tangent to any torsion
section over $\CC^0$.  In the latter case,
Theorem~\ref{thm:discreteness} is obviously true, so we may assume
from now on that $P$ does not lie in the Betti foliation.

We now define
\[ T_{Betti}:=\left\{t\in \CC^0|P\text{ is tangent to some
    }\LL_{r,s}\text{ over }t\right\}.\]

The preceding paragraph shows that $T_{tor}\cap\CC^0\subset T_{Betti}$
and so, as explained above, to prove Theorem~\ref{thm:discreteness} it
will suffice to prove that $T_{Betti}$ is a discrete subset of $\CC$.
More formally:

\begin{claim}\label{claim}
  For every $t\in \CC$, there is a classical open neighborhood $U_t$
  of\ \ $t$ in $\CC$ such that $(U_t\setminus\{t\})\cap
  T_{Betti}=\emptyset$.  In other words, for every $(r,s)\in\R^2$, $P$ is not
  tangent to $\LL_{r,s}$ over $U_t\setminus\{t\}$. 
\end{claim}

To establish Claim~\ref{claim}, we will consider cases according to the
reduction type of $\EE$ at $t$.  We use the standard Kodaira notation
($I_n$, $I_n^*$,\dots) to index the cases.

\smallskip\noindent
\textbf{The case of $I_0$ reduction:} Using diagrams~\eqref{eq:triv1}
and \eqref{eq:triv2}, we identify the section $P$ over $U$ with the
graph of a function $\phi:\Delta\to\R^2$.  Write $(r_0,s_0)=\phi(0)$
for the image of $P$ over $t$.  Since $P$ is assumed not to be
contained in $\LL_{r_0,s_0}$, we may shrink $U$ so that $P$ meets
$\LL_{r_0,s_0}$ only over $t$, in other words, so that the only value
of $z$ with $\phi(z)=(r_0,s_0)$ is $z=0$.

It is clear that $P$ is tangent to some $\LL_{r,s}$ over $t'$ if and
only if the derivative of $(x,y)\to(r,s)$ (as a map of 2-manifolds)
vanishes at the $z$ corresponding to $t'$.

To finish, we claim that after possibly
shrinking $U$ and $\Delta$, the derivative of $\phi$ does not
vanish away from $0\in \Delta$.  To see this, apply the Lojasiewicz
gradient inequality (\cite{Lojasiewicz64}, \cite{BierstoneMilman88})
to the components of $(\phi_1,\phi_2)$ of $\phi$: That result says that after
shrinking $\Delta$, there are constants $C>0$ and $0<\theta<1$ such
that
\[\left|\nabla \phi_i(z)\right|\ge C\left|\phi_i(z)-\phi_i(0)\right|^\theta\]
for all $z\in\Delta$.  But if $z\in\Delta\setminus\{0\}$,
$\phi(z)\neq\phi(0)$ so one of the $\phi_i(z)\neq\phi_i(0)$ which
implies that $\nabla \phi_i(z)\neq0$ and so the derivative of $\phi$
is also non-zero.

This establishes Claim~\ref{claim} at points of good reduction: if $t$ is such
a point, there is an open neighborhood $U_t$ of $t$ in $\CC$ such that
$P$ is not tangent to any $\LL_{r,s}$ over any $t'\in U_t\setminus\{t\}$.
To finish the proof, we deal with tangencies near places of
bad reduction.

\smallskip\noindent
\textbf{The case of $I_1$ reduction:}
We next consider the case of multiplicative reduction with an
irreducible special fiber.  I.e., assume that $\EE$ has reduction type
$I_1$ over $t\in\CC$.  Let $\Delta\to\CC$ be a holomorphic
parameterization of a neighborhood of $t$, where $\Delta$ is the unit
disk and $0\in\Delta$ maps to $t$.  Again, over the course of the
proof we will reduce the radius of $\Delta$ but not change the
notation. Let $\XX\to\Delta$ be the pull-back of $\EE\to\CC$ to
$\Delta$, let $\Delta'=\Delta\setminus\{0\}$, and let $\XX'\to\Delta'$
be the restriction of $\XX\to\Delta$ to $\Delta'$.  Note that the
special fiber
\[\XX\setminus\XX' = \text{nodal cubic} \cong \C^\times\cup \{q\}\]
where $q$ is the node of the cubic.  

According to Kodaira \cite[pp.~596ff]{Kodaira63}, we may shrink
$\Delta$ and choose $\Delta\to\CC$ so that $\XX'$ has the form
\[\XX'\cong \left(\Delta'\times\C^\times\right)/\Z\]
where the action of $\Z$ on $\Delta'\times\C^\times$ is
\[m\cdot(z,w)=(z,z^mw).\]

Moreover, as explained starting in the last paragraph of
\cite[p.~597]{Kodaira63}, there is a holomorphic map
\[\phi:\Delta\times\C^\times \to \XX\]
such that $\{0\}\times\C^\times$ maps biholomorphically to the
complement of $q$ in the special fiber, and
$\Delta'\times\C^\times\to\XX'\subset\XX$ is the natural quotient
map.  We may thus identify the section $P$ with a holomorphic map
$f:\Delta\to\C^\times$.  

Now let $V\subset\Delta'$ be a non-empty, connected and simply
connected subset, and choose a branch of the logarithm $\log:V\to\C$.
Then the local Betti leaf $\LL_{r,s}$ over $V$ (computed with respect to the
period $\tau(z)=(1/2\pi i)\log z$) is the image of the map $V\to\XX'$
which sends $z$ to the class of $(z,e^{r\log z}e^{2\pi is})$.  The
local Betti leaf $\LL_{r,s}$ through $(z,w)$ has 
\[r=\frac{\log|w|}{\log |z|}\qquad\text{and}\qquad
    s=\frac{1}{2\pi i}\log\left(\frac{w}{|w|}\right).\] 
  (Note that the logarithms appearing in the expression for $r$
  are evaluated at real numbers, so we use the standard real
  logarithm, and the class of $s$ in $\R/\Z$ is independent of the
  choice of logarithm.)

Then we calculate that $P$ is tangent to $\LL_{r,s}$ at $z\in V$ if
and only if
\begin{equation}\label{eq:deriv}
f'(z)=\left(z\mapsto e^{r\log z}e^{2\pi is}\right)'(z)=
\frac{f(z)}{z}\frac{\log|f(z)|}{\log|z|}.
\end{equation}
Note that the expression on the right is well defined independently of
the choice of $V$ and the logarithm. 

Now we assume that there is a sequence of tangencies accumulating at
$t$ and derive a contradiction.  More precisely, assume that there is
a sequence $z_i\in\Delta'$ tending to 0 such that for all $i$,
\begin{equation}\label{eq:deriv2}
f'(z_i)=\frac{f(z_i)}{z_i}\frac{\log|f(z_i)|}{\log|z_i|}.
\end{equation}
We will show that there is no holomorphic function satisfying such
equalities. 

If $f'$ is identically zero, then $f$ is a non-zero constant.
Equation~\eqref{eq:deriv} shows that the constant value of $f$ must
have absolute value 1, so $f(z)=e^{2\pi is}$ for some real $s$, and we
find that $P$ lies in a local Betti leaf $\LL_{0,s}$, in contradiction
to our assumption.

Now assume that $f'$ is not identically zero, so $f$ takes its value
at $z=0$ to finite order $N=\ord_{z=0}(f(z)-f(0))\ge1$.  Then
$g(z)=zf'(z)/f(z)$ is holomorphic on $\Delta$, and
Equation~\eqref{eq:deriv2} says that
$g(z_i)=(\log|f(z_i)|)/(\log|z_i|)$ for all $i$.  
Shrinking $\Delta$
if necessary, we have estimates
\begin{equation}\label{eq:g-estimate}
B_1|z|^N< |g(z)|<B_2|z|^N
\end{equation}
for some positive constants $B_1$ and $B_2$ and all $z\in\Delta'$.
Shrinking $\Delta$ again if necessary, we may write $f(z)=w_0(1+h(z))$
with 
\[C_1|z|^N<|h(z)|<C_2|z|^N\]
for some positive constants $C_1$ and $C_2$ and all $z\in\Delta'$.  
We have
\[\log|1+h(z)|\le\log(1+|h(z)|)\le|h(z)|< C_2|z|^N\]
for $z\in\Delta'$.
If $|w_0|=1$ (so $\log|w_0|=0$), we have
\[\left|\frac{\log|f(z)|}{\log|z|}\right|=
\left|\frac{\log|1+h(z)|}{\log|z|}\right|\le
\left|\frac{C_2|z|^N}{\log|z|}\right|.\]
Taking $z_i$ close to zero and noting that
$g(z_i)=(\log|f(z_i)|)/(\log|z_i|)$ we get a contradiction to the
lower bound in Equation~\eqref{eq:g-estimate}.

To finish, assume that $|w_0|\neq1$.  Then
\[\left|\frac{\log|f(z)|}{\log|z|}\right|\ge
\left|\frac{\log|w_0|}{\log|z|}\right|-
\left|\frac{\log|1+h(z)|}{\log|z|}\right|\ge
\left|\frac{\log|w_0|}{\log|z|}\right|-
\left|\frac{C_2|z|^N}{\log|z|}\right|.\]
Taking $z_i$ close to zero and noting that
$g(z_i)=(\log|f(z_i)|)/(\log|z_i|)$  we get a contradiction to the
upper bound in Equation~\eqref{eq:g-estimate}.

This establishes that there is no accumulation of
tangencies between $P$ and local Betti leaves at $t$ when $\EE$
has reduction of type $I_1$ at $t$.

\smallskip\noindent
\textbf{The case of $I_b$ reduction:} 
Now consider the case of multiplicative reduction of type
$I_b$ over $t\in\CC$.  This case is very similar to the $I_1$ case,
with some notational complications.

Let $\Delta\to\CC$ be a holomorphic
parameterization of a neighborhood of $t$, where $\Delta$ is the unit
disk and $0\in\Delta$ maps to $t$.  Again, over the course of the
proof we will reduce the radius of $\Delta$ but not change the
notation. Let $\XX\to\Delta$ be the pull-back of $\EE\to\CC$ to
$\Delta$, let $\Delta'=\Delta\setminus\{0\}$, and let $\XX'\to\Delta'$
be the restriction of $\XX\to\Delta$ to $\Delta'$.  Then the
special fiber $\XX\setminus\XX'$ has the form
\[\XX\setminus\XX' = \text{chain of $b$ copies of $\P^1$} 
\cong \bigcup_{i\in\Z/b\Z}\C_i^\times\cup \{q_i\}\]
where the $q_i$ are the nodes of the chain.  Let 
\[\XX^{sm}=\XX\setminus\{q_1,\dots,q_b\}\]
be the smooth locus of $\XX\to\Delta$.

Kodaira \cite[pp.~599ff]{Kodaira63}, gives a covering of $\XX^{sm}$ by
$b$ open sets as follows: for $i\in\Z/b\Z$, let
\[W_i=W_i'\cup\C^\times_i,\qquad W_i'=\left(\Delta'\times\C^\times\right)/\Z\]
where the action of $\Z$ on $\Delta'\times\C^\times$ is
\[m\cdot(z,w)=(z,z^{bm}w).\]  For
$z\in\Delta'$ and $w\in\C^\times$, write $(z,w)_i$ for the
class of $(z,w)$ in $W_i'$.  Then $\XX^{sm}$ is obtained by
glueing the  $W_i$ according to the rule
\[(z,w)_i=(z,z^{j-i}w)_j\]
for all $z\in\Delta'$, $w\in\C^\times$, and $i,j\in\Z/b\Z$.  Thus
$\XX_z$, the fiber of $\XX^{sm}\to\Delta$ over $z\neq0$, is the elliptic
curve $\C^\times/z^{b\Z}$ and the fiber over $z=0$ is a disjoint union
of $b$ copies of $\C^\times$, one appearing in each open set $W_i$.

Now assume that the section $P$ meets the special fiber at
$w_0\in\C^\times_i$.  Then we may choose a small disk $D$ around $w_0$
in $\C^\times_i$ and shrink $\Delta$ so that the image of 
\[\Delta\times D\into W_i\into \XX\]
contains the image of $P$ over $\Delta$.  We may then identify $P$
with a function $f:\Delta\to D$, and the conditions on $f$ for $P$ to
be tangent to a local Betti leaf are the same as they are in
the $I_1$ case.  Thus the rest of the argument is essentially
identical to that in the $I_1$ case, and we will omit the rest of the
details.

\smallskip\noindent
\textbf{The case of $I_b^*$ reduction:} 
Now consider the case where $\EE$ has reduction of type
$I_b^*$ at $t$.  Choose as usual a parameterization $\Delta\to\CC$ of
a neighborhood of $t$ and let $\XX\to\Delta$ be the pull-back of
$\EE\to\CC$.  Let $\tilde\Delta\to\Delta$ be a double cover ramified
at $0\in\tilde\Delta$ and let
$\tilde\Delta'=\tilde\Delta\setminus\{0\}$, so that
$\tilde\Delta'\to\Delta'$ is an unramified double cover.  Then it is
well known that $\tilde\XX'$, the pull-back of $\XX\to\Delta$ to
$\tilde\Delta'$ has an extension to $\tilde\XX\to\tilde\Delta$ whose
fiber over $0$ is of type $I_{2b}$.  Moreover, the section $P$ of
$\XX\to\Delta$ induces a section $\tilde P$ of
$\tilde\XX\to\tilde\Delta$.  We apply the argument of the previous
section to conclude that after shrinking $\tilde\Delta$, there are no
points of $\tilde\Delta'$ over which $\tilde P$ is tangent to a local
Betti leaf.  Since $\tilde\XX'\to\XX'$ is \'etale, the same must be
true after shrinking $\Delta$, i.e., $P$ is not tangent to a local
Betti leaf over $\Delta'$.  (It is clear from the definition that
local Betti leaves are preserved under an \'etale base change.)  This
proves the desired discreteness near a point where $\EE$ has $I_b^*$
reduction.

\smallskip\noindent
\textbf{The cases of $II$, $II^*$, $III$, $III^*$, $IV$, and $IV^*$ reduction:} 
Finally, consider the
cases where $\EE$ has additive and potentially good reduction.  Then
by an argument parallel to that of the previous case, we may focus
attention on a disk $\Delta$ near $t$, pull-back to a ramified cover
$\tilde\Delta\to\Delta$ of order 2, 3, 4, or 6, and reduce to the case
of good reduction.  We leave the details as an exercise for the
reader.

This completes the proof that the set of points $t\in\CC^0$ over which
$P$ is tangent to a local Betti leaf is discrete and therefore
finite, and it concludes the proof of Theorem~\ref{thm:discreteness}. 
\qed

\begin{rem}
  This proof gives no bounds on the cardinality of $T_{tor}$ or
  $T_{Betti}$, i.e., on the number tangencies.  In a sequel to this
  paper, we will give an explicit upper bound on the number of
  tangencies.  In fact, by exploiting the global Betti foliation, we
  will give an \emph{exact formula} for tangencies counted with
  multipicities and taking into account the behavior of $P$ at the bad
  fibers.  The resulting upper bound depends only on topological
  properties of $\EE$.
\end{rem}

\begin{rem}
  A referee points out that case $I_0$ of the argment above also
  follows directly from Lemma~6.4 in \cite{BedfordLyubichSmillie93}.
\end{rem}

\section{Interlude on moduli of elliptic curves with a differential
  and a point}\label{s:moduli}
In this section, we discuss certain moduli spaces of elliptic curves
with additional structure.  These spaces will be useful when we
consider families of elliptic surfaces in the following section.  We
work in more generality than needed in this paper, and readers who are
so inclined may replace the base ring $R$ below with a field $k$ of
characteristic $\neq2,3$ or even with $\C$.

We begin by noting that there is a standard model for an elliptic
curve $E$ equipped with a non-zero differential $\omega$ and a
non-trivial point $P$: Given the data, choose a Weierstrass model of
$E$
\[y^{\prime2}+a'_1x'y'+a'_3y'=x^{\prime3}+a'_2x^{\prime2}+a'_4x'+a'_6\]
such that $\omega=dx'/(2y'+a'_1x'+a'_3)$.  Then there is a unique
change of coordinates $x'=x+r$, $y'=y+sx+t$ such that $P$ has
coordinates $(x,y)=(0,0)$ and $a_1=0$.  Thus there is a unique triple
$(a_2,a_3,a_4)$ such that $E$ is the elliptic curve defined by
\[y^2+a_3y=x^3+a_2x^2+a_4x,\]
the differential is $\omega=dx/(2y+a_3)$, and the point is $P=(0,0)$.

We want to formalize this observation.  Following Deligne
\cite{Deligne75}, we say that a \emph{curve of genus 1} over a base
scheme $S$ is a proper, flat, finitely presented morphism
\[\pi:\WW\to S\]
whose geometric fibers are reduced and irreducible curves of
arithmetic genus 1 equipped with a section $O:S\to\WW$ whose image is
contained in the locus where $\pi$ is smooth.  

Let $R=\Z[1/6]$ and consider the stack $\MM$ over $\spec R$ whose
value on an $R$-scheme $S$ is the set of triples $(\WW\to S,\omega,P)$
where $\WW\to S$ is a curve of genus 1 over $S$ as defined above,
$\omega$ is a nowhere vanishing section of $O^*(\Omega^1_{\WW/S})$,
and $P:S\to\WW$ is a section disjoint from $O$.  Two such triples
$(\WW\to S,\omega,P)$ and $(\WW'\to S,\omega',P')$ are isomorphic if
there exists an $S$-isomorphism $\WW\to\WW'$ carrying $\omega$ to
$\omega'$ and $P$ to $P'$.

\begin{prop}\label{prop:moduli}
  The stack $\MM$ is represented by the affine scheme $\spec
  R[a_2,a_3,a_4]$.  The
universal object over $\MM$ is the projective family of plane cubics
$\WW\to \spec R[a_2,a_3,a_4]$
defined by 
\[y^2+a_3y=x^3+a_2x^2+a_4x\]
 equipped with the
differential $\omega=dx/(2y+a_3)$ and the section $P$ given by $x=y=0$.
The substack of $\MM$ where the curve $\WW\to S$ has smooth fibers is
represented by the open subscheme where $\Delta\neq0$
and the substack where the fibers of $\WW\to S$ are either smooth or
nodal is represented by the open subscheme where either
$\Delta\neq0$ or $2^4a_2^2-2^43a_4\neq 0$. %$c_4\neq0$.
\end{prop}

More formally, ``the projective family of plane cubics defined by
$y^2+a_3y=x^3+a_2x^2+a_4x$'' is defined as follows: Let $\RR$ be the
graded $R[a_2,a_3,a_4]$-algebra
\[\RR=R[a_2,a_3,a_4][x,y,z]/(y^2z+a_3yz^2-x^3-a_2x^2z-a_4xz^2)\]
where $x$, $y$, and $z$ have weight 1.  Then 
$\WW=\proj_{\spec R[a_2,a_3,a_4]}(\RR)$.

Here and later in the paper, whenever we have elements $a_2$, $a_3$,
$a_4$ in some ring, we set
\begin{align*}
  c_4(a_2,a_3,a_4)&=16a_2^2-48a_4\\
&=2^4a_2^2-2^43a_4\\
c_6(a_2,a_3,a_4)&=288a_2a_4-64a_2^3-216a_3^2\\
&=2^53^2a_2a_4-2^6a_2^3-2^33^3a_3^2\\
\Delta(a_2,a_3,a_4)
&=-16a_2^3a_3^2+16a_2^2a_4^2+72a_2a_3^2a_4-27a_3^4-64a_4^3\\
&=-2^4a_2^3a_3^2+2^4a_2^2a_4^2+2^33^2a_2a_3^2a_4-3^3a_3^4-2^6a_4^3.
\end{align*}
We often omit the $a_i$ and simply write $c_4$, $c_6$, or $\Delta$.
However, in the proof just below, we do not omit the $a_i$, i.e., we
distinguish between the elements $c_4$, $c_6$ generating a two-variable
polynomial ring $R[c_4,c_6]$ and the elements $c_4(a_2,a_3,a_4)$ and
$c_6(a_2,a_4,a_6)$ in the ring $R[a_2,a_3,a_4]$.

\begin{proof}[Proof of Proposition~\ref{prop:moduli}]
By \cite[Prop.~2.5]{Deligne75}, the stack of pairs $(\WW\to S,\omega)$
as above is represented by the affine scheme $\spec R[c_4,c_6]$ with
universal curve
\[y^{\prime2}=x^{\prime3}-\frac{c_4}{2^43}x'-\frac{c_6}{2^53^3}\] 
and universal differential $dx'/2y'$.  (Deligne uses the more
traditional coordinates $g_2=c_4/(2^23)$ and $g_3=c_6/(2^33^3)$, but
this is immaterial since $1/6\in R$.)  Define a morphism
$\spec R[a_2,a_3,a_4]\to\spec R[c_4,c_6]$ by sending
\[c_4\mapsto c_4(a_2,a_3,a_4)\quad\text{and}\quad 
c_6\mapsto c_6(a_2,a_3,a_4).\]
Then pulling back the universal curve over $\spec R[c_4,c_6]$ to $\spec
R[a_2,a_3,a_4]$ and making the change of coordinates $x'=x+a_2/3$,
$y'=y+a_3/2$ yields the curve and differential mentioned
in the statement of the theorem. 

To finish the proof, one checks that the fibers of $\spec
R[a_2,a_3,a_4]\to\spec R[c_4,c_6]$ are the \emph{affine} plane curves
\[y^{\prime2}=x^{\prime3}-\frac{c_4}{2^43}x'-\frac{c_6}{2^53^3},\]
i.e., $\spec R[a_2,a_3,a_4]$ is the universal curve over $\spec
R[c_4,c_6]$ minus its zero section.
Indeed, the fiber over $(c_4,c_6)$ is
\begin{align*}
c_4&=16a_2^2-48a_4\\
c_6&=288a_2a_4-64a_2^3-216a_3^2. 
\end{align*}
Eliminating $a_4$ and dividing by $2^53^3$, we find
\[\frac{a_3^2}{2^2}=\frac{a_2^3}{3^3}-\frac{c_4a_2}{2^43^2}-\frac{c_6}{2^53^3}.\]
Thus setting $a_3=2y'$ and $a_2=3x'$ yields the stated fiber.

This means that to give a morphism to $\spec R[a_2,a_3,a_4]$ is to
give a morphism to $\spec R[c_4,c_6]$ (i.e., a family of curves and a
differential) together with a non-zero point in each fiber.  This
completes the proof that $\spec R[a_2,a_3,a_4]$ represents $\MM$.

The assertions about the locus where $\WW$ has good or nodal fibers
follows from \cite[Prop.~5.1]{Deligne75}, and this completes the proof
of the proposition.
\end{proof}

In light of the proposition, from now we change notation and let $\MM$
be defined as the scheme $\spec R[a_2,a_3,a_4]$.  Also, we write
$\MM^{sm}$ for the locus where $\Delta\neq0$ and $\MM^{n}$ for the
locus where $\Delta=0$ and $c_4\neq0$.  Similarly, let $\NN=\spec
R[c_4,c_6]$, $\NN^{sm}$ the locus where $c_4^3-c_6^2\neq0$, and
$\NN^{n}$ the locus where $c_4^3-c_6^2=0$ and $c_4\neq0$.

\subsection{Torsion}\label{ss:torsion}
Let $\pi:\WW\to\MM$ be the universal curve.  Then the smooth locus of
$\pi$ is a commutative group scheme over $\MM$ and we may speak of
points of finite order in the fibers.  For each $n>1$, let $\MM[n]$ be
the locus where $P$ has order dividing $n$, let $\MM[n]'$ be the
locus where $P$ has order exactly $n$, and let
$\MM^{sm}[n]=\MM^{sm}\cap\MM[n]$ and
$\MM^{sm}[n]'=\MM^{sm}\cap\MM[n]'$.

Let $n>1$ and let $k$ be a field of characteristic zero or prime to $6n$.  For
$R$-schemes, write $-\tensor k$ for the base change along the unique
morphism $\spec k\to\spec R$.  Then it follows from \cite[I.6 and
II.1.18-20]{DeligneRapoport73} that $\MM[n]\tensor k$ is locally
closed in $\MM\tensor k$, everywhere regular and of codimension 1, and that
$\MM^{sm}[n]\tensor k$ is a divisor in $\MM^{sm}$ which is
\'etale and finite of degree $n^2$ over $\NN^{sm}$.

In fact, there are explicit recursive equations for divisors
$\DD_n\subset\MM$ such that $\MM^{sm}[n]=\MM^{sm}\cap \DD_n$, namely
the ``division polynomials'' evaluated at $P$
\cite[Ex.~3.7]{SilvermanAEC}.  More precisely, for each $n>1$, there
is a homogenous polynomial $\psi_n$ in $a_2,a_3,a_4$ (where $a_i$ has
weight $i$) of degree $n^2-1$ such that $\DD_n$ is defined by $\psi_n$.
We have
\begin{align*}
  \psi_2&=a_3,\\
\psi_3&=a_2a_3^2-a_4^2,\\
\psi_4&=2a_2a_3^3a_4-2a_3a_4^3-a_3^5,
\end{align*}
and the higher $\psi_n$ are defined recursively by
\begin{align*}
  \psi_{2m+1}&=\psi_{m+2}\psi_m^3
-\psi_{m-1}\psi_{m+1}^3\quad&m\ge2,\\
\psi_2\psi_{2m}&=\psi_{m-1}^2\psi_m\psi_{m+2}
-\psi_{m-2}\psi_m\psi_{m+1}^2\quad&m\ge3.
\end{align*}

\subsection{Nodal cubics with a point}
Let $k$ be a field of characteristic zero or $p>3$, and let
$a=(a_2,a_3,a_4)$ be a $k$-valued point of $\MM^n$, i.e., such that
$\Delta(a)=0$ and $c_4(a)\neq0$.  Then by
Proposition~\ref{prop:moduli}, the plane cubic
\[E_a:\qquad y^2+a_3y=x^3+a_2x^2+a_4x\] 
over $k$ is nodal.  We further assume that $(a_3,a_4)\neq(0,0)$ so that
$P=(0,0)$ and the node, call it $Q$, are distinct.  Let $\G_m$ be the
multiplicative group over $k$.  Then, possibly after extending $k$
quadratically, there is a group isomorphism
\[ E_a\setminus\{Q\}\to\G_m\]
which is unique up to pre-composing with inversion.  We want to write
down an explicit expression for the image of $P$ under such an
isomorphism.

This is a straightforward calculation:  The node is defined by
the vanishing of $2y+a_3$ and $3x^2+2a_2x+a_4$, and one finds that its
coordinates are
\[ Q=\left(\frac{18a_3^2-8a_2a_4}{c_4},\frac{-a_3}{2}\right)\]
where as usual $c_4=16a_2^2-48a_4$.
Changing coordinates 
\[ x=x'+\frac{18a_3^2-8a_2a_4}{c_4},\qquad
y=y'+\frac{-a_3}{2}\]
brings $E_a$ into the form
\[y^{\prime 2}=x^{\prime3}
+\frac{-c_6}{4c_4}x^{\prime2}\]
where as usual $c_6=288a_2a_4-64a_2^3-216a_3^2$.
Letting $\gamma$  be a square root of $-c_6/(4c_4)$, the map to $\G_m$
is 
\[(x',y')\mapsto\frac{y'-\gamma x'}{y'+\gamma x'}\]
and we find that $P$ maps to
\begin{equation}
  \label{eq:P-coord}
\frac{a_3c_4-\gamma(16a_2a_4-36a_3^2)}{a_3c_4+\gamma(16a_2a_4-36a_3^2)}  
\end{equation}
which (not surprisingly) is an algebraic expression in the original
$a_2,a_3,a_4$.

\section{From $E/K$ to $\EE\to\CC$}\label{s:constructing-EE}
We remind the reader how to go from an elliptic curve over a function
field to an elliptic surface.  Although this is not strictly necessary
for our main purposes, it suggests a fruitful point of view on
finite-dimensional families of elliptic surfaces parameterized by
certain Riemann-Roch spaces.

\subsection{General construction}\label{ss:construction}
Let $k$ be a field of characteristic 0 or $p>3$, let $\CC$ be a
smooth, projective, absolutely irreducible curve over $k$, and let
$K=k(\CC)$.  Let $E$ be an elliptic curve over $K$ equipped with a
non-zero rational point $P\in E(K)$.

Choose a non-zero differential $\omega$ on $E$.  Then by
Proposition~\ref{prop:moduli}, there is a unique triple
$a=(a_2,a_3,a_4)$ of elements of $K$ such that $E$ is isomorphic to
\[y^2+a_3y=x^3+a_2x^2+a_4x,\] 
$P$ is $(0,0)$ and $\omega=dx/(2y+a_3)$.  Let $D$ be the smallest
divisor on $\CC$ such that $\dvsr(a_i)+iD$ is effective for $i=2,3,4$.
(Here ``smallest'' is with respect to the usual partial ordering:
$D_1\ge D_2$ if $D_1-D_2$ is effective.)  Let $L=\O_\CC(D)$ so that
we may regard $a_i$ as a global section of $L^{\tensor i}$.

If $U\subset\CC$ is a non-empty Zariski open subset and $\phi$ is a
trivialization of $L$ over $U$ (i.e., a nowhere vanishing section of
$L$), then over $U$ we may regard the $a_i$ as functions, and we
get a morphism $U\to\MM$.  Pulling back the universal curve gives a
family
\[\WW_U\to U\]
of curves of genus 1 (in the sense used before
Proposition~\ref{prop:moduli}) with a section $P_U$ disjoint from $O$,
and the general fiber of $\WW_U\to U$ is $E/K$ equipped with $P$.  
If $\{U_j\}$ is an open cover with
trivializations $\phi_j$ of $L_{|U_j}$, there is a unique way to glue
over the intersections compatible with the identification of the
generic fiber of $\WW_{U_j}\to U_j$ with $E/K$, and the result is a
global family $\WW\to\CC$ of curves of genus 1
equipped with a section which we again denote by $P$.
Writing $\PP$ for the $\P^2$ bundle over $\CC$ given by
\[\PP=\P_{\CC}\left(L^2\oplus L^3\oplus\O_\CC\right)\]
(with coordinates $[x,y,z]$ on the fibers), we see that $\WW$ is the
closed subset of $\PP$ defined by the equation 
\[y^2+a_3y=x^3+a_2x^2+a_4x\]
and $P$ is the section $[0,0,1]$.
%The choice of $\omega$ defines a
%(possibly rational) section of $L$ whose divisor is $D$.

The surface $\WW$ may have isolated singularities, and if so, we
resolve them and then blow down any remaining $(-1)$-curves in the
fibers of the map to $\CC$, thus obtaining a smooth, relatively
minimal elliptic surface $\EE\to\CC$ with a section again denoted by
$P$.

\subsection{A geometric subtlety}
There is a subtle point hiding in the last step of this construction:
The section $P$ of $\WW\to\CC$ is disjoint from $O$,
yet a section of $\EE\to\CC$ may very well meet $O$.  Therefore, there
may be some blowing down in the last step to force such an intersection.
We make a few more comments about this situation and then give an
example.

The underlying issue is that the local models $\WW_U\to U$ are in a
sense minimal with respect to pairs ``elliptic fibration + nowhere zero
section,'' but they may not be minimal if we forget the section.  We
can quantify this as follows:  Given $E/K$ and $P$, choosing $\omega$
leads to coefficients $a_i\in K$ and to invariants
\[c_4=2^4(a_2^2-3a_4)\quad\text{and}\quad 
c_6=2^53^2a_2a_4-2^6a_2^3.\]
Recall that $D$ was defined as the smallest divisor on $\CC$ such that
$\dvsr(a_i)+iD\ge0$ for $i=2,3,4$.  Similarly, let $D'$ be the
smallest divisor on $\CC$ such that $\dvsr(c_j)+jD'\ge0$ for
$j=2,4$.  Then it is clear that $D\ge D'$ and the points entering into
$D-D'$ are exactly those where the model $\WW\to\CC$ is not minimal (in
the sense of \cite[p.~816]{SilvermanAEC}).  Moreover, while $\WW$
sits naturally as a divisor in
\[\PP=\P_{\CC}\left(L^2\oplus L^3\oplus\O_\CC\right),\]
the minimal Weierstrass family associated to $\WW\to\CC$ is naturally
a divisor in
\[\PP'=\P_{\CC}\left(L^{\prime2}\oplus L^{\prime3}\oplus\O_\CC\right)\]
where $L'=\O_\CC(D')$.  The choice of $\omega$ defines (possibly
rational) sections of $L$ and $L'$ with divisors $D$ and $D'$
respectively.  Since $O^*(\Omega^1_{\EE/\CC})=L'$, in some
sense  $L'$ is more natural than $L$.

\subsection{An example}
Let $\CC=\P^1$ and $K=k(t)$, and let $E/K$ be defined by 
\[w^2=z^3+t^2z-1\]
with point $P=(t^{-2},t^{-3})$ and differential $\omega=dz/2w$.  The
standard model coming from Proposition~\ref{prop:moduli} for this data is
\[y^2+2t^{-3}y=x^3+3t^{-2}x^2+(3t^{-4}+t^2)x\]
with $P=(0,0)$ and $\omega=dx/(2y+2t^{-3})$.  Also, $c_4=-48t^2$ and
$c_6=864$ and we find that 
\[D=0+\infty\quad\text{and}\quad D'=\infty.\]
The local model $\WW_{\A^1}\to\A^1$ is given by
\[y^2+2y=x^3+3x^2+(3+t^6)x.\]
The fiber over $t=0$ is a cubic with cusp at $t=0$, $x=y=-1$, and the
surface $\WW_{\A^1}$ is singular at this point.  Resolving
the singularity requires blowing up once and normalizing, and a
further blow down removes a $(-1)$-curve in the fiber.  This last blow
down brings the section $P$ into contact with the zero section $O$. 

\subsection{Starting with the line bundle}\label{ss:L-to-EE}
We take the following point of view on constructing elliptic surfaces
over $\CC$:  Start with a line bundle $L$ on $\CC$.  Then for each 
\[a=(a_2,a_3,a_4)\in H^0(\CC,L^2\oplus L^3\oplus L^4)\]
with $\Delta(a_2,a_3,a_4)\neq0$,
we get $\WW\to\CC$ defined by the vanishing of
\[y^2z+a_3yz^2=x^3+a_2x^2z+a_4xz^2\]
in 
\[\PP=\P_{\CC}\left(L^2\oplus L^3\oplus\O_\CC\right).\]
For ``most'' choices of $a$, $\WW\to\CC$ is already minimal and
$L'=L$.  This holds if $\Delta(a_2,a_3,a_4)$ has order of vanishing
$<12$ (as a section of $L^{12}$) at each place of $\CC$.  If $\Delta$
has only simple zeroes, then $\WW\to\CC$ is minimal and $\WW$ is
regular, so $\EE=\WW$.  In this way, we get flat families of elliptic
surfaces parameterized by open subsets of certain Riemann-Roch spaces.
We will justify the claim ``most'' in the next section.

% remark on flatness:
% Given $\CC$ and $L$, cover $\CC$ with $U_i$ so that $L$ is
% trivialized.  Then over each $U_i$,
% $\P_\CC(L^2\oplus L^3\oplus L^4)\to\CC$ is just $P^2_{U_i}\to
% U_i$. Let $V^o$ be the subset of $H^0(\CC,L^2\oplus L^3\oplus L^4)$
% where $\Delta$ has distinct zeroes, and let
% $\WW\subset \P_\CC(L^2\oplus L^3\oplus L^4)$ be the universal surface.
% Then over $V^o\times U_i$, $\WW$ is a family of reduced and
% irreducible plane cubics.  It is thus flat by Hartshorne III.9.9.  By
% local nature of flatness, $\WW\to V^o\times\CC$ is flat.  Since
% $V^o\times\CC\to V^o$ is obviously flat, $\WW\to V^o$ is flat.  Also,
% $\WW\to V^o$ has smooth geometric fibers, and it is flat, so by
% Hartshorne III.10.2 it is a smooth morphism of relative dimension 2.

\section{Very general elliptic surfaces with two
  sections}\label{s:v-g-surfaces}
In this section, $k$ is a field of characteristic zero or $p>3$ and
$\CC$ is a smooth, projective, absolutely irreducible curve over $k$.
Let $L$ be a line bundle on $\CC$ which is globally generated and
write $d$ for the degree of $L$.

Let $a=(a_2,a_3,a_4)$ be an element of
$V=H^0(\CC,L^2\oplus L^3\oplus L^4)$ with $\Delta(a)\neq0$.  Then as
explained in Section~\ref{ss:L-to-EE} we get a family $\WW_a\to\CC$ of
curves of genus 1 and a relatively minimal elliptic surface
$\EE_a\to\CC$ equipped with a section $P$.  Our aim is to show that for
a very general choice of $a$, $P$ is transverse to $\EE_a[n]$ for all
$n$ and enjoys other desirable properties.

We first consider the case where $d=0$, so $L$ is trivial and the
$a_i$ are constants.  In this case, it is clear that $P$ is transverse
to all torsion sections if and only if it is disjoint from all torsion
sections, if and only if it is of infinite order. This happens for
very general choices of $a$, but not on a Zariski open.  That suggests
what to expect in the general case.

We restate Theorem~\ref{thm:vg} (in the case where $L$ is non-trivial)
with an additional claim:
\begin{thm}\label{thm:vg+}
  Let $L$ be a globally generated line bundle on $\CC$ of degree $d>0$,
  and set 
\[V=H^0(L^2\oplus L^3\oplus L^4).\]
Then for a very general $a=(a_2,a_3,a_4)\in V$, the elliptic surface
$\EE_a\to\CC$ associated to
\[E:\qquad y^2+a_3y=x^3+a_2x^2+a_4x\]
equipped with the section $P=(0,0)$ has the following properties:
\begin{enumerate}
\item $P$ has infinite order
\item The singular fibers of $\EE_a\to\CC$ are nodal cubics \textup{(}i.e.,
  Kodaira type $I_1$\textup{)}.
\item $P$ meets each singular fiber in a non-torsion point.
\item If $n$ is not a multiple of the characteristic of $k$, then $P$
  is transverse to $\EE_a[n]$.
\item If $n$ is not a multiple of the characteristic of $k$, then $nP$
  meets $O$ transversally in $d(n^2-1)$ points.
\end{enumerate}
\end{thm}

Here, as usual, ``for a very general $a$'' means that there is a
countable union of non-empty, Zariski open subsets of $V$ such that if
$a$ lies in their intersection, then the assertion holds for $a$.  We
will prove several lemmas, each asserting that some Zariski open
subset is non-empty, and then put them together to prove the theorem
at the end of this section.  It is no loss of generality to assume that
$k$ is algebraically closed, so for convenience we assume this for the
rest of the section.

Recall that ``$L$ is globally generated'' means that for all
$t\in \CC$, there is a global section of $L$ not vanishing at $t$.  It
is a standard exercise to show that $L$ is globally generated and has
positive degree if and only if there is a non-constant morphism
$f:\CC\to\P^1$ such that $L=f^*\O_{\P^1}(1)$.  Moreover, if $L$ is
globally generated, then the set of global sections of $L$ with
reduced divisors (i.e., distinct zeroes) is non-empty and Zariski
open.
% To see non-emptiness: If $f:\CC\to\P^1$ is
% defined by sections $s_0$ and $s_1$ via $p\mapsto[s_o(p):s_1(p)]$,
% then $a_0s_0-a_1_s_1$ has distinct zeroes if and only if $f$ is
% unramified over $[a_1:a_0]\in\P^1$. 

\begin{lemma}\label{lemma:V-Delta}
  The subset $V_\Delta\subset V$ consisting of $a$ such that
  $\Delta(a)$ has $12d$ distinct zeroes \textup{(}as a section of
  $L^{12}$\textup{)} is Zariski open and not empty.  There are
  $a\in V_\Delta$ whose zeroes are disjoint from any given finite
  subset of points of $\CC$.
\end{lemma}

\begin{proof}
  It is clear that the locus of $a\in V$ where $\Delta(a)$ has
  distinct zeroes is Zariski open.
  % this is too clumsy:
  %Indeed, $L^{12}$ is globally
  %generated, so its set of sections with distinct zeroes is Zariski
  %open, and $V_\Delta$ is the inverse image of this set under the
  %polynomial map $a\mapsto\Delta(a)$.
  To prove the lemma, we need to
  check that $V_\Delta$ is not empty.  We do this constructively.
  First assume $\CC=\P^1$ and $L=\O_{\P^1}(1)$.  Set $a_2=0$,
  $a_3=c\in k$, and $a_4=t^4$.  Then $\Delta=-27c^4-64t^{12}$ which
  has distinct zeroes as a section of $\O_{\P^1}(12)$ if $c\neq0$.
  Moreover, varying $c$, we can arrange for the zeros to avoid any
  finite subset of $\P^1$.

In the general case, choose a morphism $f:\CC\to\P^1$ such that
$L=f^*(\O_{\P^1}(1))$.  Let $S\subset\P^1$ be the branch locus of $f$.
Then setting $a_2=0$, $a_3=c$, and $a_4=f^*(t^4)$, where $c$ is chosen
so that the zeroes of $-27c^4-64t^{12}$ are disjoint from $S$, yields
an explicit $a$ with the required properties.  Varying $c$ allows us to
avoid any finite subset of $\CC$.
\end{proof}

As noted in Section~\ref{ss:L-to-EE}, if $a\in V_\Delta$, then the
corresponding elliptic surface $\WW_a$ is smooth (so no resolution
of singularities is needed), $\WW_a\to\CC$ is relatively minimal (so we
may set $\EE_a=\WW_a$), and $L=O^*(\Omega^1_{\EE_a/\CC})$.  Moreover, the
bad fibers of $\EE_a\to\CC$ are all of type $I_1$.  From now on we
always choose $a$ from $V_\Delta$.

\begin{lemma}\label{lemma:sing-fibers}
  For every $n\ge1$, there is a non-empty, Zariski open subset $V_n$ of
  $V_\Delta$ such that if $a\in V_n$, then the section $P$ of
  $\EE_a\to\CC$ does not intersect any singular fiber in a point of
  order exactly $n$.
\end{lemma}

\begin{proof}
  It is clear that the locus of $a$ where $P$ has the stated property
  is open, and our task is to show it is non-empty.  Since the bad
  fibers are all of type $I_1$, if $k$ has characteristic $p>0$ and
  $n$ is divisible by $p$, there are no points of order exactly $n$ in
  the fiber, so we may take $V_n=V_\Delta$.  

  Now assume that $n$ is not divisible by the characteristic of $k$.
  We check constructively that there is a non-empty set as described
  in the statement.  As in the previous lemma, we may reduce to the
  case $\CC=\P^1$ and $L=\O_{\P^1}(1)$.  Take $a_2=0$, $a_3=c$,
  $a_4=t^4$.  Then the bad fibers are at the roots of
  $t^{12}=(-27/64)c^4$ and at each such root, the coordinate in $\G_m$
  of $P$ was given at \eqref{eq:P-coord}.  For the data we are
  considering, the coordinate is
\[\frac{4t^4-3c\gamma}{4t^4+3c\gamma}\qquad
  \text{where $\gamma=(-9c^3/2)^{1/2}t^{-2}$}.\] 
Then for each $n$, there are only finitely many values of $c$ such
that for some root $t$ of $t^{12}=(-27/64)c^4$, the displayed quantity
is an $n$-th root of unity.  This proves that $V_n$ is non-empty for
each $n$.
\end{proof}

\begin{rem}
  Over an uncountable field, intersecting the opens in the theorem
  gives a non-empty set.  We can do a bit better over $\C$: There is
  an everywhere dense classical open set in $V_\Delta$ such that $P$
  meets each singular fiber away from the unit circle
  $S^1\subset\C^\times$.
\end{rem}

\begin{lemma}\label{lemma:no-p-torsion}
  If the characteristic of $k$ is $p>3$, then for all $a\in V_\Delta$
  and any $n$ divisible by $p$, $P$ does not have order exactly $n$.
\end{lemma}

\begin{proof}
  It will suffice to show that when $a\in V_\Delta$, $\EE_a\to\CC$ has
  no non-trivial $p$-torsion sections.  First note that since $a\in
  V_\Delta$, the zeroes of $c_4$ are disjoint from those of $\Delta$.
  This implies that $j=c_4^3/\Delta$ has simple poles, so it is not a
  constant (implying that $\EE_a\to\CC$ is non-isotrivial) and not a
  $p$-th power.  Then \cite[Prop~I.7.3]{Ulmer11} implies that
  $\EE_a\to\CC$ has no $p$-torsion.  (In \cite{Ulmer11}, the ground field
  is finite, but the argument there works over any field of
  positive characteristic.)
\end{proof}

\begin{prop}\label{prop:torsion}
  For every $n$ not divisible by the characteristic of $k$, there is a
  non-empty, Zariski open subset $W_n\subset V_\Delta$ such that if
  $a\in W_n$, then $nP\neq0$ and $P$ is transverse to $\EE_a[n]$. 
\end{prop}

\begin{proof}
  Again, it is clear that the set of $a$ with the desired properties is
  open.  Unfortunately, it seems hopeless to give a constructive proof
  that it is non-empty, so we have to do something more sophisticated.

Recall the moduli space $\MM$ of Section~\ref{s:moduli}.  We write
$\MM_k$ for 
\[\MM\tensor_{\Z[1/6]}k=\spec k[a_2,a_3,a_4]\]
and $\MM_k[n]$ for the locally closed, smooth,
codimension 1 locus parameterizing triples $(E,\omega,P)$ where $P$
has order $n$.

Recall also that $V=H^0(\CC,L^2\oplus L^3\oplus L^4)$ and $V_\Delta$
is the open subset consisting of $a$ such that $\Delta(a)$ has
distinct zeroes.  Choose an open subset $U\subset\CC$ and a
trivialization of $L$ over $U$.  Then for $a=(a_2,a_3,a_4)$ the $a_i$
may be regarded as functions on $U$, and we get a morphism
$f_a:U\to\MM_k$.  To say that $nP=O$ is to say that $f_a(U)$ is
contained in $\MM_k[n]$.  To say that $P$ is tangent to $\EE_a[n]$
over $x\in U$ is to say that $f_a(U)$ is tangent to $\MM_k[n]$ at
$f_a(x)$.  We will show that these conditions do not hold for most
$a$.

Consider the morphism
\[F: V\times U\to\MM_k\qquad (a,t)\mapsto F(a,t):=f_a(t)\] 
and let 
\[D_n:=F^{-1}(\MM_k[n])\cap\left(V_\Delta\times U\right).\] 
We will use the global generation of $L$ to show that $D_n$ is a
smooth, locally closed subset of codimension 1 in $V_\Delta\times U$,
and that there is a non-empty open subset $W_{U,n}\subset V_\Delta$
such that the projection $D_n\to V_\Delta$ is \'etale over $W_{U,n}$.
This means that if $a\in W_{U,n}$, then $\{a\}\times U$ is transverse
to $D_n$, i.e., that $P$ meets the $n$-torsion multisection of $\EE_a$
transversally over $U$.  Taking a finite cover $\{U_j\}$ of $\CC$ and
setting $W_n=\cap_jW_{U_j,n}$ will complete the proof.

% previous proof was not complete in characteristic p:
% we need to check that $D_n$ is separable over $V_\Delta$
% what's below is a modified proof using the same ideas that would
% have been needed to do that.

Since $L$ is globally generated, so are its powers $L^i$ for
$i=2,3,4$.  This means that for every $t\in U$, there are global
sections $a_2,a_3,a_4$ not vanishing at $t$, and for all but finitely
many $t$ there are global sections $s_2,s_3,s_4$ which vanish to order
1 at $t$.  (Since $L^i$ is globally generated, there are sections of
$L^i$ inducing a morphism $\CC\to\P^1$.  If $t$ is not in the
ramification locus, a section $s_i$ as above can be obtained by
pulling back a section of $\OO_{\P^1}(1)$ vanishing simply at the
point of $\P^1$ under $t$.)

For each $t\in U$, the restriction
\[F_t:V\times\{t\}\to\MM_k\] 
is a linear map, and since $L$ is globally generated, it is
surjective.  Thus the fibers are all affine spaces of dimension $h-3$
where $h=\dim V$.  Therefore, $F$ is surjective and smooth (smooth
because it is submersive, i.e., it has a surjective differential at
every point).  Moreover, the fibers of $F$ are $\A^{h-3}$-bundles over
$U$, and in particular, they are all irreducible of dimension
$h-2$.  It follows that each irreducible component $D_{n,i}$ of $D_n$
is smooth and locally closed in $V_\Delta\times U$ of codimension 1
and has the form 
\[D_{n,i}=F^{-1}(\MM_k[n]_i)\cap\left(V_\Delta\times U\right)\] 
where $\MM_k[n]_i$ is an irreducible component of $\MM_k[n]$.

Consider an irreducible component $D_{n,i}$ of $D_n$.  We are going to
produce a point of $D_{n,i}$ at which the projection $D_{n,i}\to V$ is
\'etale.  Start by choosing any point $(a,t)\in D_{n,i}$ and let
$m=F(a,t)$.  The fiber of $F$ over $m$ is an $\A^{h-3}$ bundle over
$U$, and $F^{-1}(m)\cap(V_\Delta\times U)$ is a non-empty open subset
of this bundle, so it projects to a non-empty open subset of $U$.
This means that we may find another point
$(a',t')=(a'_2,a'_3,a'_4,t')$ in $D_{n,i}$ such that $L^i$ admits
global sections $s_i$ vanishing simply at $t'$ for $i=2,3,4$.

For all triples $(\alpha_2,\alpha_3,\alpha_4)\in k^3$, we have 
\[F(a'_2+\alpha_2s_2,a'_3+\alpha_3s_3,a'_4+\alpha_4s_4,t)=m,\] 
and for a non-empty open subset of triples $(\alpha_i)\in k^3$, we
have that \[a'':=(a'_2+\alpha_2s_2,a'_3+\alpha_3s_3,a'_4+\alpha_4s_4)\in
V_\Delta.\]  
By a suitable choice of the $\alpha_i$
we may arrange for the differential of $F$ restricted to
$\{a''\}\times U$ to carry the tangent space of $U$ at $t$ to a line in
the tangent space of $\MM_k$ at $m$ not contained in $T_{\MM[n],m}$.
For such a choice, we conclude that $\{a''\}\times U$ is transverse to
$D_{n,i}$ at $(a'',t)$.  This proves that the projection
$D_{n,i}\to V_\Delta$ is \'etale at $(a'',t)$.

It follows that there is a Zariski open subset $D_{n,i}^o$ of
$D_{n,i}$ such that $D_{n,i}^o\to V_\Delta$ is \'etale.  The image of
$D_{n,i}\setminus D_{n,i}^o$ in $V_\Delta$ is contained in a proper
closed subset, and removing these subsets for all $i$ yields an open
subset $W_{U,n}$ over which $D_n\to V_\Delta$ is \'etale.  Covering
$\CC$ with finitely many $U_j$ and setting $W_n=\cap_jW_{U_j,n}$
yields an open subset of $V_\Delta$ such that if $a\in W_n$, then $P$
does not have order $n$ and is transverse to $\EE_a[n]$.  This
completes the proof of the proposition.
\end{proof}

\begin{proof}[Proof of Theorem~\ref{thm:vg+}]
Consider the intersection 
\[V'=\left(\bigcap_{n\ge1} V_n\right)\bigcap 
\left(\bigcap_{p\,\nodiv\,n}W_n\right)\subset V_\Delta.\] 
The preceding lemmas show that if $a\in V'$, then the corresponding
$\EE_a$ has the properties asserted in the Theorem.  Indeed, since
$a\in V_\Delta$, $\Delta(a)$ as $12d$ distinct zeroes, and so $\EE_a$
has $12d$ bad fibers of type $I_1$ and no other bad fibers.  Since
$a\in\cap_n V_n$, Lemma~\ref{lemma:sing-fibers} shows that $P$ does
not meet a bad fiber in a torsion point.  Since $a\in\cap_n W_n$,
Lemma~\ref{lemma:no-p-torsion} and Proposition~\ref{prop:torsion} show
that $P$ has infinite order, and Proposition~\ref{prop:torsion} shows
that if $n$ is prime to the characteristic, then $P$ is transverse to
$\EE_a[n]$.  This establishes points (1) through (4) of the Theorem.

The transversality in point (5) is equivalent to that in (4), so to
finish we just need to calculate the intersection multiplicity $(nP).O$.
For this, we first note that $P.O=0$ by construction, and as explained
in the proof of Lemma~\ref{lemma:int-numbers}, $O^2=P^2=-d$.  Thus
$ht(P)=2d$ and Lemma~\ref{lemma:int-numbers} implies that
$(nP).O=d(n^2-1)$, as required.

This completes the proof of the theorem.
\end{proof}

\begin{rem}
  When $\CC=\P^1$, every line bundle of non-negative degree is
  globally generated.  Thus, starting from data $(E,P)$ over $\P^1$,
  we can find a deformation $(\EE',P')$ with the same base $\CC$ and
  bundle $L$ such that $P'$ is transverse to all torsion
  multisections.  For a general $\CC$, if we do not assume any
  positivity for $L=O^*(\Omega^1_{\EE/\CC})$, it may be impossible to
  produce deformations with fixed $\CC$ and $L$. Here are two
  alternatives: First, we may embed $L\into L'$ where $L'$ is globally
  generated, and deform a non-minimal model of $\EE$ (lying in
  $\P_\CC(L^{\prime2}\oplus L^{\prime 3}\oplus\OO_\CC)$).  Second, it
  seems likely that the ideas of Moishezon \cite{Moishezon77}, as
  explained in \cite[Thm.~I.4.8]{FriedmanMorganSFMCS} would allow one
  to find a deformation of $\EE$ where the base curve is also allowed
  to vary (i.e., deform to $\EE'\to\CC'$ and section $P'$) with the
  desired transversality.
\end{rem}

\begin{rem}
  Suppose that $k$ has characteristic zero and that $\pi:\EE\to\CC$
  and $P$ satisfy the conclusions of Theorem~\ref{thm:vg+}.  If $n_1$
  and $n_2$ are two distinct integers, then $n_1P\cup n_2P$ is a
  normal crossings divisor on $\EE$.  More generally, if $N\subset\Z$
  is a non-empty finite set, then
  \[D=\bigcup_{n\in N}nP\] 
  is a curve on $\EE$ with only ordinary multiple points.  Indeed, it
  is a union of smooth components which meet pairwise transversally.
% To see this,
% suppose $n_1<\cdots< n_k$ and that $n_1P, n_2P,\dots,n_kP$  meet at
% a point of $x\in\EE$ over $t\in\CC$.  Translating by $-n_1P$, the
% local picture is the same as that of $O, (n_2-n_1)P,\dots,(n_k-n_1)P$.
% If $m$ is the smallest positive integer such that $mP$ meets $O$ over
% $t$, then the $(n_j-n_1)$ for $j=1,\dots,k$ are distinct multiples of
% $m$ and our claim is that $O, (n_2-n_1)P, \dots,(n_k-n_1)P$ have
% distinct tangent directions at the point of $O$ over $t$.  Let
% $v_O\in T_{\EE,x}$ be a vector tangent to $O$ and let $v_O+v_t$ be a
% vector tangent to $mP$ where $v_t$ is tangent to the fiber of $\pi$
% over $t$. That such a $v_t$ exists and is non-zero follows from the
% fact that$mP$ is transverse to $O$.  Finally, the discussion in
% Section~\ref{ss:mult-by-n} on the action of multiplication by $n$
% shows that $v_O+(1/m)(n_j-n_1)v_t$ is tangent to $(n_j-n_1)P$.  Since these
% vectors are distinct for distinct values of $j$, this shows that the
% branches of $D$ cross pairwise transversally at $x$, as claimed.
This is clear from the facts that $O$ and $nP$ meet transversally for
all $n\neq0$ and that $O\cup(n_2-n_1)P$ is
carried isomorphically to $n_1P\cup n_2P$ under translation by $n_1P$.
\end{rem}

\section{Explicit examples with even height over small
  fields}\label{s:explicit} 
In this section, we show by explicit construction that there are pairs
$(\EE,P)$ with $P$ transverse to torsion multisections over fields $k$
such as number fields and global function fields.  The precise
statement is Theorem~\ref{thm:explicit} in the introduction.  For
simplicity, we assume throughout that the characteristic of $k$ is not
2.  We begin by constructing examples of height 2 over $\P^1$.

\begin{prop}\label{prop:explicit}
  Let $k$ be a field of characteristic $\neq2$.  Then there exist
  Jacobian elliptic surfaces $\EE\to\P^1$ over $k$ equipped with a
  section $P$ such that 
\begin{enumerate}
\item $P$ has infinite order.
\item The singular fibers of $\EE\to\P^1$ are of Kodaira type $I_0^*$.
\item $P$ meets each singular fiber in a non-torsion point.
\item If $n$ is not a multiple of the characteristic of $k$, then $nP$
  meets $O$ transversally in 
 \[\begin{cases}
\displaystyle\frac{n^2-1}2&\text{if $n$ is odd,}\\
\displaystyle\frac{n^2-4}2&\text{if $n$ is even,}
\end{cases}\]
points.
\item The height of $\EE$ is 2, i.e., $O^*(\Omega^1_{\EE/\P^1})\cong\OO_{\P^1}(2)$.
\end{enumerate}
\end{prop}

\begin{proof}
  We will construct one such $\EE\to\P^1$ for every elliptic curve $E$
  over $k$.  Suppose that $f\in k[x]$ is a monic polynomial of degree
  3 such that $E$ is defined by $y^2=f(x)$.  Form the product
  $E\times_k E$, and let $\{\pm1\}\subset\aut(E)$ act diagonally.
  The quotient $(E\times_k E)/(\pm1)$ is a singular (Kummer) surface,
  and projection to the first factor induces a morphism 
\[(E\times_k E)/(\pm1)\to E/(\pm1)\cong\P^1.\] 

Let $\EE\to\P^1$ be
  the regular minimal model of $(E\times_k E)/(\pm1)\to\P^1$.  Thus
  $\EE$ is obtained from $(E\times_k E)/(\pm1)$ by blowing up the 16
  fixed points of $\pm1$ on $E\times_k E$, and the bad fibers of
  $\EE\to\P^1$ are of type $I_0^*$ and lie over $t=\infty$ and the
  roots of $f(t)$.  

Let $\Gamma_n\subset E\times_k E$ be the graph of multiplication by
$n$, which we may regard as the image of a section to
$E\times_k E\to E$.  Then $\Gamma_n$ is preserved by $\pm1$ and maps
with degree 2 to a section of $(E\times_k E)/(\pm1)\to\P^1$, and this
section lifts to a section $nP$ of $\EE\to\P^1$.  (The notation is
consistent in that $nP$ is $n$ times $P$ in the group law of $\EE$.)

The following diagram summarizes the data:
\[
  \xymatrix{
    E\times_k E\ar[r]\ar[d]&(E\times_k E)/(\pm1)\ar[d]&\EE\ar[l]\ar[d]\\
E\ar@/^1pc/^{\Gamma_n}[u]\ar[r]
&E/(\pm1)\cong\P^1\ar@{=}[r]&\P^1.\ar@/_1pc/_{nP}[u]}
\]

It will be convenient to have a Weierstrass equation for $\EE\to\P^1$.  If
$f(x)=x^3+ax^2+bx+c$, then $\EE$ is the N\'eron model of the elliptic
curve 
\[ y^2=x^3+af(t)x^2+bf^2(t)x+cf^3(t)\]
over $k(t)$, and the point $P$ has coordinates $(x,y)=(tf(t),f^2(t))$.
Indeed, if the two factors of $E\times E$ are $v^2=f(u)$ and
$s^2=f(r)$, then the field of invariants of $\pm1$ is generated by
$u$, $r$, and $z=vs$, and these satisfy the equation 
\[z^2=f(u)f(r).\]
Setting $u=t$ and $z=y/f(u)$, and $r=x/f(u)$ yields the equation and
point above.

We now verify the cases $n=1$ and $n=2$ of the proposition.  Since $P$
has polynomial coefficients, it does not meet $O$ over any finite
value of $t$, and since its $x$ and $y$ coordinates have degrees 4 and
6, and $\EE$ has height 2, $P$ also does not meet $O$ over $t=\infty$.
In summary, $P$ meets $O$ nowhere, as claimed.  For later use, we
note that at the roots of $f$, $P$ specializes to $(0,0)$, i.e., to
a singular point of the fiber of $(E\times_k E)/(\pm1)$, so $P$
lands on a non-identity component of the fiber of $\EE$.  At
$t=\infty$, $P$ specializes to $(1,1)$, a non-singular, finite point of
the fiber (i.e., a point not on $O$).

A tedious but straightforward calculation (or an algebra package \dots)
shows that $2P$ has coordinates $((1/4)t^4+\cdots,(1/8)t^6+\cdots)$
where $\cdots$ indicates terms of lower degree in $t$.  The argument of
the previous paragraph shows that $2P$ meets $O$ nowhere, as
claimed.  For later use, we note that $2P$ passes through a
finite point of the identity component in each of the bad fibers.

Now consider $n>2$.  It is clear that $\Gamma_n$ meets $E\times\{0\}$
exactly at the points of $E$ of order $n$, and each of these
intersections is transverse.  If $(p,0)$ is such a point which is not
of order 2, then the quotient map
\[E\times_k E\to(E\times_k E)/(\pm1)\] 
is \'etale in a neighborhood of $(p,0)$ and it sends $\Gamma_n$ 2-to-1
to a curve that meets $O$ transversally.  Moreover, the map
\[\EE\to(E\times_k E)/(\pm1)\]
is an isomorphism in a neighborhood of such a point.  This proves that
$nP$ meets $O$ transversally over the values of $t$ such that there is
a point $(t,v)$ with $v^2=f(t)$ which is $n$-torsion and not
2-torsion.  There are
\[
  \begin{cases}
\frac{n^2-1}2&\text{if $n$ is odd}\\
\frac{n^2-4}2&\text{if $n$ is even}    
  \end{cases}\]
such values of $t$.  

It remains to consider what happens over the roots of $f(t)$ and
$t=\infty$.  But we checked above that $P$ meets a non-trivial point
of the identity component at $t=\infty$ and such a point is either of
infinite order or of order $p$ when $k$ has characteristic $p$.  So, for
$n$ prime to the characteristic of $k$, $nP$ does not meet $O$ over
$t=\infty$.  Similarly, over the roots of $f(t)$, $P$ passes through
the non-identity component and $2P$ passes through a non-trivial point
of the identity component, so $nP$ does not meet $O$ when $n$ is prime
to the characteristic.  We have thus identified all points where $nP$
and $O$ intersect, the intersections are transverse, and their number
is as stated in the proposition.  This completes the proof of the
proposition.
\end{proof}

\begin{rem}\label{rem:height}
  As a check, we compute the intersection number $(nP).O$ using
  heights as in Lemma~\ref{lemma:int-numbers}.  We have $O^2=P^2=-2$
  and $P.O=0$.  Since $P$ passes through a non-identity component of
  the fibers over roots of $f(t)$ and through the identity component
  at $t=\infty$, the ``correction term'' is $-C_P.(P-O)=-3$.
(See table 1.19 in \cite{CoxZucker79}.)  Using the formula
\eqref{eq:ht-def} for the height pairing yields $ht(P)=1$.

Similarly, for any odd $n$, $-C_{nP}.(nP-O)=-3$ and using that
$ht(nP)=n^2$ and calculating as in Lemma~\ref{lemma:int-numbers} we
find $(nP).O=(n^2-1)/2$.

On the other hand, for even $n$, $nP$ passes through the identity
component in all bad fibers, so $-C_{nP}.(nP-O)=0$ and we find that 
$(nP).O=(n^2-4)/2$.

This confirms that the intersections we saw above are all transverse.
\end{rem}

\begin{proof}[Proof of Theorem~\ref{thm:explicit}]
  Proposition~\ref{prop:explicit} implies the case of the Theorem
  where $\CC=\P^1$ and $L=\OO_{\P^1}(2)$, and we get infinitely many
  examples because $k$ is infinite.  Indeed, for each $j\in k$, there
  is an elliptic curve $E$ with $j$-invariant $j$, and elliptic curves
  with distinct $j$-invariants give rise to non-isomorphic
  $\EE\to\P^1$ since the non-singular fibers are twists of the chosen
  $E$.
  
  We deduce the general case by a pull-back construction.  Write
  $\EE'\to\P^1$ for one of the surfaces constructed in
  Proposition~\ref{prop:explicit}.  Let $f:\CC\to\P^1$ be a
  non-constant morphism defined by sections of the globally generated
  line bundle $F$, so $F=f^*\OO_{\P^1}(1)$ and $L=f^*\OO_{\P^1}(2)$.
  The conclusions of the theorem will hold for $\EE:=f^*\EE'\to\CC$ if
  the branch locus of $f$ is disjoint from the set of points of $\P^1$
  over which $\EE'$ has bad reduction or $nP$ meets $O$.  From the
  construction of $\EE'$, we see that the set to be avoided is
  precisely the set of $x$ coordinates of torsion points of the
  elliptic curve $y^2=f(x)$ used to construct $\EE'$.  Although this
  set is infinite, we will see that it is sparse in $k$.

  We divide into two cases according to the characteristic of $k$,
  starting with the case of characteristic zero.  Choose an elliptic
  curve $E$ over $\Q$, and an auxiliary prime $\ell$ such that
  equations defining $E$ are $\ell$-integral and $E$ has good
  reduction modulo $\ell$. Then \cite[VIII.7.1]{SilvermanAEC} implies
  that the $x$-coordinate of a torsion point $Q$ (defined over some
  number field $K$ and taken with respect to an $\ell$-integral model)
  is ``almost integral,'' i.e., it satisfies $\ell^2 x(Q)$ is integral
  at all primes of $K$ over $\ell$.  Construct $\EE'\to \P^1_\Q$ using
  $E$ as in Proposition~\ref{prop:explicit}.  Then choose any
  non-constant morphism $f:\CC\to\P^1_k$ defined by sections of $F$.
  Composing $\phi$ with a linear fractional transformation, we may
  arrange that the branch locus of $f$ consists of points with finite,
  non-zero coordinates, and that any of those coordinates which lie in
  a number field have large denominators at primes over $\ell$.  They
  are thus distinct from the $x$-coordinates of torsion points of $E$,
  and $\EE=f^*\EE'$ satisfies the requirements of the theorem.

  When $k$ has characteristic $p>2$, the argument is similar, but
  simpler: Choose an embedding $\Fp(t)\into k$, an elliptic curve $E$
  over $\Fp(t)$, and a place $v$ of $\Fp(t)$ where $E$ has good
  reduction.  Then by \cite[\S4]{Tate75}, the coordinates of any
  torsion point $Q$ of $E$ (defined over some algebraic extension $K$
  of $\Fp(t)$ and taken with respect to an integral model) are
  integral at places of $K$ over $v$.  Use $E$ to construct $\EE'$ as
  in Proposition~\ref{prop:explicit}.  Then choose any non-constant
  morphism $f:\CC\to\P^1_k$ defined by sections of $F$.  Composing
  $\phi$ with a linear fractional transformation, we may arrange that
  the branch locus of $f$ consists of points with finite, non-zero
  coordinates, and that any of those coordinates which are algebraic
  over $\Fp(t)$ are not integral at places over $v$.  They are thus
  distinct from the $x$-coordinates of torsion points of $E$, and
  $\EE=f^*\EE'$ satisfies the requirements of the theorem.
\end{proof}

% thesis project!
\begin{rem}
  It seems likely that when $k$ is a number field or a global function
  field, the construction in Proposition~\ref{prop:explicit} gives
  rise to elliptic divisibility sequences $D_n$ whose ``new parts''
  $D'_n$ are often irreducible, i.e., prime divisors.
\end{rem}

\section{Application to geography of surfaces}\label{s:geography}
In this section, we will prove Theorem~\ref{thm:geography}.  Let
$k=\C$, $\CC=\P^1$, and $L=\O_{\P^1}(d)$ where $d=g+1$, which by
assumption satisfies $d\ge1$.  Theorem~\ref{thm:vg} guarantees the
existence of an elliptic surface $\pi:\EE\to\P^1$ of height $d$ (i.e.,
such that $O^*(\Omega^1_{\EE/\P^1})=L$) with a section $P$ such that for
all $n$, $nP$ meets $O$ transversally in $d(n^2-1)$ points.  Moreover,
$\pi$ has irreducible fibers.  Let $F$ be the class of a fiber of
$\pi$. We have $O^2=P^2=-d$, $F^2=0$, and the canonical divisor of
$\EE$ is
\[K_\EE=(d-2)F.\]  
Thus the geometric genus of $\EE$ is $d-1=g$.

Fix an integer $n>1$.  Later in the proof, we will need to assume that
$n$ is sufficiently large.  Let $h:Y\to\EE$ be the result of blowing
up all but one of the points of intersection of $O$ and $nP$, let
$E_i$ ($i=1,\dots,d(n^2-1)-1$) be the exceptional divisors, and let
$C_j$ be the strict transform the section $jP$.

Write $\tilde F$ for the strict transform of a general fiber of $\pi$ in $Y$.
We have
\[C_0^2=C_n^2=-dn^2+1,\quad C_0.C_n=1,\quad\text{and}\quad
K_Y=(d-2)\tilde F+\sum_i E_i.\]

It is a simple exercise to check that the intersection pairing on $Y$
is negative definite on the lattice spanned by $C_0$ and $C_n$, and
that $p_a(Z)\le0$ for all effective divisors supported on
$C_0\cup C_n$. Thus by Artin's contractibility theorem
\cite[Thm.~2.3]{Artin62} or \cite[Thm~3.9]{BadescuAS}, we may contract
$C_0\cup C_n$.  In other words, there is a proper birational morphism
$f:Y\to X$ where $X$ is a normal, projective surface,
$f(C_0\cup C_n)=\{x\}$, and $f$ induces an isomorphism
\[Y\setminus(C_0\cup C_n)\cong X\setminus\{x\}.\]

\begin{proof}[Proof that $X$ satisfies the conditions of
  Theorem~\ref{thm:geography}]
  We have already observed that $X$ is normal and projective.  Since
  the geometric genus is a birational invariant, and $\EE$ has
  geometric genus $g$, so does $X$.  

  It is evident that $X$ has exactly one singular point, namely $x$,
  and the minimal resolution of $x$ is the union of two smooth
  rational curves ($C_0$ and $C_n$) meeting at one point and having
  self-intersection $-a:=-dn^2+1$.  Such a singularity is analytically
  equivalent to a cyclic quotient singularity of type
  $1/(a^2-1)(1,a)$, as one sees by considering the Hirzebruch-Jung
  continued fraction
\[a-\frac1a=\frac{a^2-1}a.\]
(See \cite[Ch.~3]{BarthHulekPetersVandeVenCCS}.)  In particular, it
follows that $X$is $\Q$-Gorenstein and $K_X$ is $\Q$-Cartier.

We next compute the discrepancy of $x$ (as defined for example in
\cite{KawamataMatsudaMatsuki87}) and verify that $x$ is log-terminal.  
Since $C_j$ is smooth and rational with self-intersection
$-a$, we have $C_j.K_Y=a-2$.  Define coefficients $\alpha_0,\alpha_n\in\Q$ by
\[K_Y=f^*K_X+\alpha_0C_0+\alpha_nC_n\]
(an equality in $\Pic(Y)\tensor\Q$).  Then
\begin{align*}
  0&=(f_*C_0).K_X\\
&=C_0.f^*K_Y\\
&=(a-2)+\alpha_0a-\alpha_n
\end{align*}
and similarly,
\[0=(a-2)-\alpha_0+\alpha_na.\]
We find that 
\[\alpha_0=\alpha_n=-\frac{a-2}{a-1}>-1.\]
This confirms that $x$ is a log-terminal singularity, and we have
\[f^*K_X=(d-2)\tilde F+\sum_iE_i+\frac{a-2}{a-1}\left(C_0+C_n\right).\]

Next, write $b:=(a-2)/(a-1)$ and compute
\begin{align*}
K_X^2&=(f^*K_X)^2\\
&=\left((d-2)\tilde F+\sum_i E_i + b(C_0+C_n)\right)^2\\
&=dn^2(4b-2b^2-1)+d+1+4b^2-12b.
\end{align*}
As $n\to\infty$, $a\to\infty$ and $b\to1$, so $K_X^2$ grows like
$dn^2$ and in particular is unbounded as $n$ varies.

To finish the proof, it remains to check that $K_X$ is ample, which we
do using the Nakai-Moishezon criterion \cite[Thm.~1.22]{BadescuAS}.
We have already seen that $K_X^2>0$, so it will suffice to check that
for every irreducible curve $C$ on $X$, $C.K_X>0$.  For any such curve
$C$, 
\[f^*C=D+m_0C_0+m_nC_n\] 
where $D$ is an irreducible curve not equal to $C_0$ or $C_n$, and
$m_j\ge0$ for $j=0,n$.  It will thus suffice to prove that
$D.f^*K_X>0$ for all irreducible curves not equal to $C_0$ or $C_n$
and that $C_j.f^*K_X=0$ for $j=0,n$.  For the latter assertion, one
computes that
\[C_j.f^*K_X=f_*(C_j).K_X=0\]
for $j=0,n$.

For the former assertion, we make a case by case analysis of the
possibilities for $D$.  They are:
\begin{itemize} 
\item the strict transform $\tilde F$ of a general fiber of $\pi$, for
  which we have 
\[\tilde F.f^*K_X=2(a-2)/(a-1)>0;\]
\item one of the exceptional curves $E_i$, for which we have
\[E_i.f^*K_X=-1+2(a-2)/(a-1),\] 
which is $>0$ if $d>1$ or $n>2$;
\item the strict transform $\tilde G_i=\tilde F-E_i$ of a fiber of
  $\pi$ passing though an intersection point of $O$ and $nP$, for
  which we have $\tilde G_i.f^*K_X=1$;
\item and the strict transform $\tilde Q$ of a multisection $Q$ of
  $\pi$ not equal to $O$ or $nP$.  Let $e$ be the degree of
  $\pi_{|Q}:Q\to\P^1$, assume that $n>2$ so that $b>1/2$, and recall
  that
\[f^*K_X=(d-2)\tilde F+\sum_iE_i+b\left(C_0+C_n\right).\]
If $d>2$,  we have
\[\tilde Q.f^*K_X\ge (d-2)e>0.\]
If $\tilde Q.\sum_iE_i>e$ or $\tilde Q.C_0>2e$, then again it is clear that
$\tilde Q.f^*K_X>0$ as required.  To finish, assume that $d\le2$,
$\tilde Q.\sum_iE_i\le e$, and $\tilde Q.C_0\le 2e$.  Applying $h^*$
to the equality in Lemma~\ref{lemma:nP-in-pic} implies that 
\[C_n=nC_1+(1-n)C_0-n\sum_iE_i+d(n^2-n)\tilde F\]
in $\NS(Y)$.  If $Q\neq P$,
we find that
\[\tilde Q.C_n\ge(1-n)2e-ne+d(n^2-n)e\]
which is $>e$ for all $n\ge4$, and this shows that
$\tilde Q.f^*K_X>0$.  If $Q=P$, then
\[\tilde Q.C_n\ge  -nd-n+d(n^2-n)\]
and we find that $\tilde Q.f^*K_X>0$ for all $n\ge5$.  (When $Q=P$,we
can also calculate directly that $\tilde Q.f^*K_X=d-2+b(dn^2-2dn)$
which goes to infinity with $n$.)
This completes the check that $\tilde Q.f^*K_X>0$ for all irreducible
multisections $\tilde Q$ not equal to $C_0$ or $C_n$.
\end{itemize}
The itemized list completes the verification that $K_X$ is ample, and
this finishes the proof of the theorem. 
\end{proof}

\bibliography{database}

\def\cprime{$'$}
\begin{bibdiv}
\begin{biblist}

\bib{AlexeevLiu19b}{article}{
      author={Alekseev, V.~A.},
      author={Liu, V.},
       title={On accumulation points of volumes of log surfaces},
        date={2019},
     journal={Izv. Ross. Akad. Nauk Ser. Mat.},
      volume={83},
       pages={5\ndash 25},
}

\bib{AlexeevLiu19c}{article}{
      author={Alexeev, V.},
      author={Liu, W.},
       title={Log surfaces of {P}icard rank one from four lines in the plane},
        date={2019},
     journal={Eur. J. Math.},
      volume={5},
       pages={622\ndash 639},
}

\bib{AlexeevLiu19a}{article}{
      author={Alexeev, V.},
      author={Liu, W.},
       title={Open surfaces of small volume},
        date={2019},
     journal={Algebr. Geom.},
      volume={6},
       pages={312\ndash 327},
}

\bib{Alexeev94}{article}{
      author={Alexeev, V.},
       title={Boundedness and {$K^2$} for log surfaces},
        date={1994},
     journal={Internat. J. Math.},
      volume={5},
       pages={779\ndash 810},
}

\bib{AlexeevMori04}{incollection}{
      author={Alexeev, V.},
      author={Mori, S.},
       title={Bounding singular surfaces of general type},
        date={2004},
   booktitle={Algebra, arithmetic and geometry with applications ({W}est
  {L}afayette, {IN}, 2000)},
   publisher={Springer, Berlin},
       pages={143\ndash 174},
}

\bib{Artin62}{article}{
      author={Artin, M.},
       title={Some numerical criteria for contractability of curves on
  algebraic surfaces},
        date={1962},
     journal={Amer. J. Math.},
      volume={84},
       pages={485\ndash 496},
}

\bib{BadescuAS}{book}{
      author={B{\u{a}}descu, L.},
       title={Algebraic surfaces},
      series={Universitext},
   publisher={Springer-Verlag},
     address={New York},
        date={2001},
}

\bib{BarthHulekPetersVandeVenCCS}{book}{
      author={Barth, W.~P.},
      author={Hulek, K.},
      author={Peters, C. A.~M.},
      author={Van~de Ven, A.},
       title={Compact complex surfaces},
     edition={Second},
      series={Ergebnisse der Mathematik und ihrer Grenzgebiete. 3. Folge. A
  Series of Modern Surveys in Mathematics},
   publisher={Springer-Verlag},
     address={Berlin},
        date={2004},
      volume={4},
}

\bib{BedfordLyubichSmillie93}{article}{
      author={Bedford, E.},
      author={Lyubich, M.},
      author={Smillie, J.},
       title={Polynomial diffeomorphisms of {${\bf C}^2$}. {IV}. {T}he measure
  of maximal entropy and laminar currents},
        date={1993},
     journal={Invent. Math.},
      volume={112},
       pages={77\ndash 125},
}

\bib{BierstoneMilman88}{article}{
      author={Bierstone, E.},
      author={Milman, P.~D.},
       title={Semianalytic and subanalytic sets},
        date={1988},
     journal={Inst. Hautes \'{E}tudes Sci. Publ. Math.},
      number={67},
       pages={5\ndash 42},
}

\bib{CorvajaDemeioMasserZannierpp}{unpublished}{
      author={Corvaja, P.},
      author={Demeio, J.},
      author={Masser, D.},
      author={Zannier, U.},
       title={On the torsion values for sections of an elliptic scheme},
        date={2019},
        note={Preprint, arxiv:1909.01253},
}

\bib{CorvajaMasserZannier18}{article}{
      author={Corvaja, P.},
      author={Masser, D.},
      author={Zannier, U.},
       title={Torsion hypersurfaces on abelian schemes and {B}etti
  coordinates},
        date={2018},
     journal={Math. Ann.},
      volume={371},
       pages={1013\ndash 1045},
}

\bib{CoxZucker79}{article}{
      author={Cox, D.~A.},
      author={Zucker, S.},
       title={Intersection numbers of sections of elliptic surfaces},
        date={1979},
     journal={Invent. Math.},
      volume={53},
       pages={1\ndash 44},
}

\bib{Deligne75}{incollection}{
      author={Deligne, P.},
       title={Courbes elliptiques: formulaire d'apr{\`e}s {J}. {T}ate},
        date={1975},
   booktitle={Modular functions of one variable, {IV} ({P}roc. {I}nternat.
  {S}ummer {S}chool, {U}niv. {A}ntwerp, {A}ntwerp, 1972)},
   publisher={Springer},
     address={Berlin},
       pages={53\ndash 73. Lecture Notes in Math., Vol. 476},
}

\bib{DeligneRapoport73}{article}{
      author={Deligne, P.},
      author={Rapoport, M.},
       title={Les sch\'{e}mas de modules de courbes elliptiques},
        date={1973},
       pages={143\ndash 316. Lecture Notes in Math., Vol. 349},
}

\bib{FriedmanMorganSFMCS}{book}{
      author={Friedman, R.},
      author={Morgan, J.~W.},
       title={Smooth four-manifolds and complex surfaces},
      series={Ergebnisse der Mathematik und ihrer Grenzgebiete (3) [Results in
  Mathematics and Related Areas (3)]},
   publisher={Springer-Verlag, Berlin},
        date={1994},
      volume={27},
}

\bib{GhiocaHsiaTucker18}{article}{
      author={Ghioca, D.},
      author={Hsia, L.-C.},
      author={Tucker, T.},
       title={A variant of a theorem by {A}ilon-{R}udnick for elliptic curves},
        date={2018},
     journal={Pacific J. Math.},
      volume={295},
       pages={1\ndash 15},
}

\bib{Ingram-etal12}{article}{
      author={Ingram, P.},
      author={Mah\'{e}, V.},
      author={Silverman, J.~H.},
      author={Stange, K.~E.},
      author={Streng, M.},
       title={Algebraic divisibility sequences over function fields},
        date={2012},
     journal={J. Aust. Math. Soc.},
      volume={92},
       pages={99\ndash 126},
}

\bib{KawamataMatsudaMatsuki87}{incollection}{
      author={Kawamata, Y.},
      author={Matsuda, K.},
      author={Matsuki, K.},
       title={Introduction to the minimal model problem},
        date={1987},
   booktitle={Algebraic geometry, {S}endai, 1985},
      series={Adv. Stud. Pure Math.},
      volume={10},
   publisher={North-Holland, Amsterdam},
       pages={283\ndash 360},
}

\bib{Kodaira63}{article}{
      author={Kodaira, K.},
       title={On compact analytic surfaces. {II}},
        date={1963},
     journal={Ann. of Math. (2)},
      volume={77},
       pages={563\ndash 626},
}

\bib{KollarShepherd-Barron88}{article}{
      author={Koll\'{a}r, J.},
      author={Shepherd-Barron, N.~I.},
       title={Threefolds and deformations of surface singularities},
        date={1988},
     journal={Invent. Math.},
      volume={91},
       pages={299\ndash 338},
}

\bib{Liu-pp17}{unpublished}{
      author={Liu, W.},
       title={The minimal volume of log surfaces of general type with positive
  geometric genus},
}

\bib{Moishezon77}{book}{
      author={Moishezon, B.},
       title={Complex surfaces and connected sums of complex projective
  planes},
      series={Lecture Notes in Mathematics, Vol. 603},
   publisher={Springer-Verlag, Berlin-New York},
        date={1977},
        note={With an appendix by R. Livne},
}

\bib{Lojasiewicz64}{unpublished}{
      author={{\L}ojasiewicz, S.},
       title={Ensembles semi-analytiques},
        date={1964},
        note={IHES Preprint},
}

\bib{Shioda90}{article}{
      author={Shioda, T.},
       title={On the {M}ordell-{W}eil lattices},
        date={1990},
     journal={Comment. Math. Univ. St. Paul.},
      volume={39},
       pages={211\ndash 240},
}

\bib{Shioda99}{incollection}{
      author={Shioda, T.},
       title={Mordell-{W}eil lattices for higher genus fibration over a curve},
        date={1999},
   booktitle={New trends in algebraic geometry ({W}arwick, 1996)},
      series={London Math. Soc. Lecture Note Ser.},
      volume={264},
   publisher={Cambridge Univ. Press},
     address={Cambridge},
       pages={359\ndash 373},
}

\bib{Silverman05}{article}{
      author={Silverman, J.~H.},
       title={Generalized greatest common divisors, divisibility sequences, and
  {V}ojta's conjecture for blowups},
        date={2005},
     journal={Monatsh. Math.},
      volume={145},
       pages={333\ndash 350},
}

\bib{SilvermanAEC}{book}{
      author={Silverman, J.~H.},
       title={The arithmetic of elliptic curves},
     edition={Second},
      series={Graduate Texts in Mathematics},
   publisher={Springer},
     address={Dordrecht},
        date={2009},
      volume={106},
}

\bib{Tate75}{incollection}{
      author={Tate, J.~T.},
       title={Algorithm for determining the type of a singular fiber in an
  elliptic pencil},
        date={1975},
   booktitle={Modular functions of one variable, iv ({P}roc. {I}nternat.
  {S}ummer {S}chool, {U}niv. {A}ntwerp, {A}ntwerp, 1972)},
   publisher={Springer},
     address={Berlin},
       pages={33\ndash 52. Lecture Notes in Math., Vol. 476},
}

\bib{Ulmer11}{incollection}{
      author={Ulmer, D.},
       title={Elliptic curves over function fields},
        date={2011},
   booktitle={Arithmetic of ${L}$-functions ({P}ark {C}ity, {UT}, 2009)},
      series={IAS/Park City Math. Ser.},
      volume={18},
   publisher={Amer. Math. Soc.},
     address={Providence, RI},
       pages={211\ndash 280},
}

\bib{Ulmer13a}{article}{
      author={Ulmer, D.},
       title={On {M}ordell-{W}eil groups of {J}acobians over function fields},
        date={2013},
     journal={J. Inst. Math. Jussieu},
      volume={12},
       pages={1\ndash 29},
}

\bib{ZannierUnlikely12}{book}{
      author={Zannier, U.},
       title={Some problems of unlikely intersections in arithmetic and
  geometry},
      series={Annals of Mathematics Studies},
   publisher={Princeton University Press, Princeton, NJ},
        date={2012},
      volume={181},
        note={With appendixes by David Masser},
}

\end{biblist}
\end{bibdiv}

\end{document}